\documentclass[a4paper,UKenglish,cleveref, autoref, thm-restate,numberwithinsect]{lipics-v2021}

\hideLIPIcs 
\nolinenumbers

\title{The stochastic block model has the overlap graph property for modularity}

\usepackage{amsmath, tikz,amssymb,color,comment}
\usepackage{needspace}
\usepackage[normalem]{ulem}
\usepackage{pgfplots}
\pgfplotsset{compat=1.18} 

\newcommand{\pr}{{\mathbb P}}
\newcommand{\cA}{\mathcal{A}}
\newcommand{\cB}{\mathcal{B}}
\newcommand{\cD}{\mathcal{D}}
\newcommand{\cP}{\mathcal{P}}

\newcommand{\cX}{\mathcal{X}}
\newcommand{\q}{q^{\star}}

\newcommand{\eps}{\varepsilon}
\newcommand{\vol}{{\rm vol}}

\newcommand{\Fc}{
}

\author{Shankar Bhamidi}{Department of Statistics and Operations Research, University of North Carolina, Chapel Hill, USA.}{bhamidi@email.unc.edu}{0009-0000-3953-936X}{Bhamidi was partially supported by NSF DMS-2113662, DMS-2413928, DMS-2434559 and NSF RTG grant DMS-2134107.}

\author{David Gamarnik}{MIT Sloan School of Management.}{gamarnik@mit.edu}{0000-0001-8898-8778}{Gamarnik partially supported by NSF Grant CISE 2233897.}

\author{Remco van der Hofstad}{Department of Mathematics and Computer Science, Eindhoven University of Technology, 
The Netherlands.}{rhofstad@win.tue.nl}{0000-0003-1331-9697}{Partially supported by the Netherlands Organisation for Scientific Research (NWO) through the Gravitation NETWORKS grant 024.002.003.}

\author{Nelly Litvak}{Department of Mathematics and Computer Science, Eindhoven University of Technology, 
The Netherlands.}{n.v.litvak@tue.nl}{0000-0002-6750-3484}{Partially supported by the Netherlands Organisation for Scientific Research (NWO) through the Gravitation NETWORKS grant 024.002.003.}

\author{Pawe{\l} Pra{\l}at} {Department of Mathematics, Toronto Metropolitan University, 
Canada.} {pralat@torontomu.ca}{0000-0001-9176-8493}{Partially supported by NSERC Discovery Grant.}

\author{Fiona Skerman\footnote{Corresponding author.}}{Department of Mathematics, Uppsala University, Sweden.}{fiona.skerman@math.uu.se}{0000-0003-4141-7059}{Partially supported by the Wallenberg AI, Autonomous Systems and Software Program (WASP) funded by the Knut and Alice Wallenberg Foundation.}

\author{Yasmin Tousinejad}{Department of Mathematics, Uppsala University, 
Sweden.}{yasmin.tousinejad@math.uu.se}{0000-0002-6860-9985}{Partially supported by the Wallenberg AI, Autonomous Systems and Software Program (WASP) funded by the Knut and Alice Wallenberg Foundation.}

\authorrunning{S~Bhamidi, D~Gamarnik, R~vd~Hofstad, N~Litvak, P~Pra{\l}at, F~Skerman, Y~Tousinejad} 

\Copyright{Shankar Bhamidi, David Gamarnik, Remco van der Hofstad, Nelly Litvak, Pawel Pralat, Fiona Skerman, Yasmin Tousinejad}



\category{} 

\relatedversion{} 

\funding{This material is based upon work supported by the National Science Foundation under Grant No. DMS-1928930, while Bhamidi, Gamarnik, van der Hofstad, Litvak, Pra{\l}at and Skerman were in residence at the Simons Laufer Mathematical Sciences Institute (formerly MSRI) in Berkeley, California, during the Spring 2025 semester.} 

\acknowledgements{We thank Subhabrata Sen and Ekaterina Toropova for useful discussions.}

\EventEditors{Sayan Bhattacharya, Danupon Nanongkai, Michael Benedikt, and Gabriele Puppis}
\EventNoEds{4}
\EventLongTitle{53rd International Colloquium on Automata, Languages, and Programming (ICALP 2026)}
\EventShortTitle{ICALP 2026}
\EventAcronym{ICALP}
\EventYear{2026}
\EventDate{July 7--10, 2026}
\EventLocation{Royal Holloway, University of London, Egham, United Kingdom}
\EventLogo{}
\SeriesVolume{374}
\ArticleNo{32}

\begin{document}
\maketitle

\begin{abstract}

    The overlap gap property (OGP) is a statement about the geometry of near-optimal solutions. Exhibiting OGP implies failure of a class of local algorithms; and has been observed to coincide with conjectured algorithmic limits in problems with statistical computational gap.

    We consider the Stochastic Block Model (SBM), where the graph has a planted partition with $k$ equal-size blocks which form the `communities', and where, for parameters $p>q$, vertices within the same community connect with probability $p$, while vertices in different communities connect with probability $q$, independently across pairs of vertices. Modularity--based clustering algorithms have become ubiquitous in applications. This article studies theoretical limits of local algorithms based on the modularity score on the SBM. 

    We establish that modularity exhibits OGP on the SBM. This rules out a class of local algorithms based on modularity for recovery in the SBM, and shows slow mixing time for a related Markov Chain.  Theoretically this is one of the few instances where OGP has been established for a `planted' model, as most such analyses to date consider the `null' model.

    As part of our analysis, we extend a result by Bickel and Chen 2009, who established that with high probability, the modularity optimal partition of SBM is~$o(n)$ local moves away from the planted partition, where $n$ is the graph size. We show that, with high probability, any partition with modularity score sufficiently near the optimal value is close to the planted partition. 
\end{abstract}

\section{Introduction}

Modularity maximisation \cite{newman2006modularity,NewmanGirvan} is one of the main methods used for finding clusters and communities in networks. Modularity maximisation is NP-hard~\cite{brandes2007modularity}; in practice, modularity has a complex landscape, and there are many high scoring solutions~\cite{good2010performance}. On the other hand, one major advantage of modularity is that its local updates are easy to compute. Naturally, local updates form the initial phase in the go-to algorithms for modularity maximisation: the Louvain algorithm~\cite{blondel2008fast} and its improved version, the Leiden algorithm~\cite{traag2019louvain}. However, as already observed in~\cite{blondel2008fast}, algorithms based on local updates get trapped in a local maximum. To avoid this, both Louvain and Leiden include an additional non-local phase. An alternative approach, proposed by Wang and Kolter~\cite{wang2020community} via the Locale algorithm, escapes local maxima by replacing the local update with low-cardinality embedding.  

Given the abundance of applications of modularity maximisation and Louvain and Leiden algorithms, theoretical understanding of modularity--based algorithms is surprisingly scarce. Cohen-Addad et al.~\cite{cohen2020power} proved that in the Stochastic Block Model (SBM) with sufficient signal and two equal-sized blocks, a local algorithm, very similar to the local phase of Louvain, initiated on a random bisection, recovers the correct partition with high probability. See Section~\ref{sec.conj} for further discussion including how this relates to our result.  

The question remains whether local updates lead to a global optimum in this random model with more than two communities or different initialisation. In this work, we answer this question negatively. Specifically, we establish that in SBM with $k \ge 3$ blocks,  modularity exhibits the {\it overlap gap property} (OGP). 
The OGP is a statement of the geometry of near-optimal solutions and is considered a signature of algorithmic hardness---see Definition~\ref{def:OGP}. 
Indeed, for many problems known to exhibit a computational-statistical gap, the threshold for OGP coincides with the conjectured algorithmic threshold~\cite{gamarnikAMSNotices,gamarnik2022disordered}. The idea has origins in statistical physics, and can be used to prove failure of greedy and Markov Chain Monte Carlo algorithms in average case problems. We show OGP in a planted random model as in~\cite{gamarnik2022sparse,gamarnik2021overlap}, though the theory is usually applied to optimisation problems in null models~\cite{gamarnikAMSNotices}.

The practical implication of this result is that in networks with more than two communities, the local updates of Louvain or Leiden can get stuck in a local maximum even for a random graph model with clear communities. The same holds for other algorithms based on local updates. For example, G\"osgens et al.~\cite{gacksgens2023hyperspherical} offer a new interpretation of modularity as an angular distance in a high-dimensional hypersphere, thus establishing equivalence of modularity maximisation and nearest-neighbour search. Then, OGP implies that greedy algorithms for nearest-neighbour search may return a {\em local} rather than the {\em global} maximum. Furthermore, using the OGP we prove that a natural generalisation of the greedy algorithm, namely the Markov Chain Monte Carlo algorithm, mixes exponentially slowly, and, in particular, takes an exponential time to reach partitions close to the planted partition with high probability.

While there are few theoretical results for modularity--based algorithms, there are positive results for recovering ground truth partitions using local-updates based on modularity in the SBM. As mentioned earlier, Cohen-Addad et al.~\cite{cohen2020power} showed that for the SBM with sufficient signal and two equal-sized blocks, a local algorithm recovers the communities with high probability. This was further extended to general $k$ by giving  an algorithm with parallel local updates based on the modularity function~\cite{cohen2022massively}. For more details and a {\Fc discussion of how these relate to our OGP results}, see Section~\ref{sec.conj}.

More is known about the behaviour of modularity on random models; we briefly review this for context.
Considering first random graph models without planted structure, the main results focus on the setting of growing average degree $\bar{d}$ for the random $d$-regular graph~\cite{lichev2022modularity,mcdiarmid2018modularity,prokhorenkova2017modularity} and the Erd\H{o}s-R\'enyi random graph~\cite{mcdiarmid2020modularity,rybarczyk2025new}, which are likely to have modularity of the order $1/\sqrt{\bar{d}}$. For the Preferential Attachment (PA) model~\cite{barabasi1999emergence}, very recently, this same behaviour was established up to log factors~\cite{rybarczyk2025modularity}.
For models with planted structure, break-through results were given for the SBM by Bickel and Chen~\cite{bickel2009nonparametric} showing that the modularity-optimal partition is close to the planted partition. These were extended to the degree-corrected SBM in~\cite{zhao2012consistency}. See Theorem~\ref{thm.SBMrecovery} for details. Finally, the modularity in the Artificial Benchmark for Community Detection (ABCD) model~\cite{kaminski2021artificial}, a model similar to the well-known LFR model~\cite{lancichinetti2008benchmark} used by practitioners, was studied in~\cite{kaminski2022modularity}.

\section{Results}

\subsection{Definitions}

Let $G = (V,E)$ be a graph with $m = |E| \geq 1$ edges and $n = |V| \ge 1$ nodes. For a partition $\cA$ of the vertices of~$G$, the \emph{modularity score} of $\cA$ on $G$ is defined as
\begin{eqnarray}
q_\cA(G) = 
\frac{1}{2m}\sum_{A\in \cA} \sum_{u,v \in A}  
\left( {\mathbf 1}_{\{uv\in E\}} - \frac{d_u d_v}{2m} \right) 
= \sum_{A \in \cA} \frac{e(A)}{m} - \sum_{A\in \cA} \left( \frac {\vol (A)}{\vol(V)} \right)^2, \label{def.mod} 
\end{eqnarray}
where $d_u$ denotes the degree of node $u$, $e(A)$ denotes the number of edges within $A\subseteq V$ and $\vol(A) = \sum_{v \in A} d_v$ denotes the (degree) volume of the set $A$. The \emph{modularity} of a graph $G$ is  defined as $\q(G)=\max_\cA q_{\cA}(G)$, where the maximum is taken over all partitions $\cA$ of the nodes of~$G$. It will be useful to express modularity as the difference of the \emph{edge-contribution} or \emph{coverage} $q^E_\cA$ and the \emph{degree-tax} $q^D_\cA$ of the partition $\cA$, defined as
\begin{equation}\label{def.qEandqD} q^E_\cA(G)=\sum_{A\in\cA}e(A)/m \quad \mbox{and} \quad q^D_\cA(G)=\sum_{A\in\cA}\vol(A)^2/\vol(G)^2.\end{equation}
The modularity score, introduced by Newman and Girvan in~\cite{NewmanGirvan}, is a quality function of many popular community detection algorithms such as Louvain~\cite{blondel2008fast} or Leiden~\cite{traag2019louvain}. Indeed, the modularity score favours partitions of the set of nodes of a graph $G$ in which a large proportion of the edges fall entirely within the partition, but benchmarks it against the expected number of edges one would see in the same partition in a corresponding Chung-Lu random graph model~\cite{chung2006complex} with expected degree sequence taken to be identical to the degree sequence of the observed graph $G$. 

Fix $k\geq 2$ denoting the number of communities. We now formally define the  \emph{stochastic block model} (SBM) which will be the focus of this paper. This model is also called the \emph{planted partition model}. The set of nodes $[n]=\{1, \ldots, n\}$ is partitioned into $k$ communities $P_1, \ldots, P_k$ of as equal sizes as possible; that is, $\lfloor n/k \rfloor \le |P_i| \le \lceil n/k \rceil$ for all $i \in [k]$. For simplicity, often we will assume that $n$ is divisible by $k$, but our main result is stated and holds without this restriction. We will refer to this \emph{planted partition} $(P_i)_{i=1}^k$ as $\cP$. The probability of observing an edge between two nodes of the same community is equal to $p$; otherwise, it is equal to $q$. We will use the notation $G \in \mathcal{G}(n,k,p,q)$ whenever a random graph $G$ is generated with this probability distribution. 

Results in the paper are in the large network limit $n \rightarrow \infty$. In particular, we will assume that both $p=p(n)$ and $q=q(n)$ are functions of $n$, and $n$ is large enough for certain statements to be true. On the other hand, the parameter $k \ge 2$ is an arbitrary but {\em fixed} integer. We emphasise that the notations $o(\cdot)$ and $O(\cdot)$ refer to functions of $n$, not necessarily positive, whose growth vanishes, respectively, is bounded. 

We will establish that modularity exhibits OGP in the SBM. To prove this, we will obtain a more detailed result that characterises the maximum modularity of any partition with given distance to the planted partition. We define the distance as the classification error---see, for example,~(4.25) in~\cite{avrachenkov2022statistical}; in case of two communities, this distance is equal to the `imbalance' in~\cite{cohen2020power}. Define the distance between a $k$-part partition $\cA = (A_i)_{i=1}^k$ and the planted partition $\cP = (P_i)_{i=1}^k$ in SBM as
\begin{equation}\label{eq:dist_meas}
d(\cA, \cP) = 1 - \frac{1}{n} \max_\sigma \sum_{i=1}^k |A_{\sigma(i)} \cap P_i|, 
\end{equation}
where the maximum is taken over all permutations $\sigma: [k]\to[k]$ of $[k] = \{1, \ldots, k\}$ that govern how parts of $\cA$ are aligned with the ones of $\cP$. 
Note that this $d(\cA, \cP)$ has a natural interpretation: it is the minimum proportion of nodes that need to be re-shuffled to transform the candidate partition $\cA$ into the ground truth or planted partition $\cP$. Indeed, for a given permutation $\sigma$, one can keep nodes in $A_{\sigma(i)} \cap P_i$ where they are and move other nodes to the appropriate parts.

Now that we have a notion of distance between partitions, we may formally define the {\em overlap gap property} (OGP).
The definition of the OGP pertains to a particular instance $G_n$ of a distribution over graphs $\mathcal{G}_n$. %

\begin{definition}[\bf Overlap gap property -- planted model]\label{def:OGP}
For a graph $G_n$ with planted partition~$\cP$, the optimisation problem $\max_\cA q_\cA(G_n) $
exhibits OGP with values $\mu > 0$ and $0 \leq \nu_1 < \nu_2$ if the following holds: %
For any partition $\cB$ for which $q_\cB(G_n)\geq \q(G_n)-\mu$, it holds that either $d(\cB, \cP)\leq \nu_1$, or that $d(\cB, \cP)\geq\nu_2$. Furthermore, $d(\cB', \cP)\geq \nu_2$ does indeed occur for some partition $\cB'$ with $q_{\cB'}(G_n)\geq \q(G_n) - \mu$. 
\end{definition}

\subsection{Modularity exhibits OGP}

Theorem~\ref{thm:OGP} below states that the modularity score exhibits OGP on the SBM, provided that the number of communities, $k$, is at least 3:

\begin{theorem}
[\bf Modularity has OGP]\label{thm:OGP} 
Fix real numbers $a>b>0$, and integer $k\geq 3$.
For any $\nu\in \left(\frac{1}{2(k-1)}, \frac{1}{k}\right)$ and {$\varepsilon>0$}, there exist $\mu = \mu(\nu)$ and $c=c(\eps)$ such that the following holds for $n$ large enough.

Let $p=p(n)=\omega\; a / n$ and $q=q(n)=\omega \; b/n$ be such that $\omega>c$, and let $G \in \mathcal{G}(n,k,p,q)$ with planted partition $\cP$. Then with probability at least $1-\eps$, for every $k$-part partition $\cA$ with 
\begin{equation}
\label{eq:OGP}
q_\cA(G) \ge \frac {a-b}{a+(k-1)b} \left( 1 - \frac {1}{k} - \frac {2}{k^2} \right) - \mu,
\end{equation} 
either $d(\cA, \cP) \le \frac {{1}}{2(k-1)}$ or $d(\cA, \cP) \ge \nu$.
Moreover, there are partitions $\cA$ satisfying the latter. 
\end{theorem}

\begin{figure}
  \centering
  \includegraphics[scale=0.2]{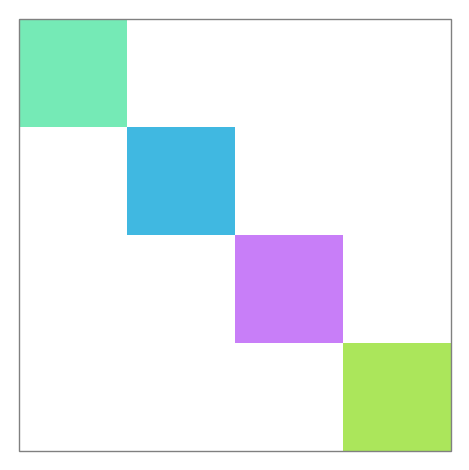} \;\;\;\;\;\;\;\;\;\;\;\;\;\;\;\;\;\;\;\;\;\;\;\;\;\;\;\;
  \includegraphics[scale=0.2]{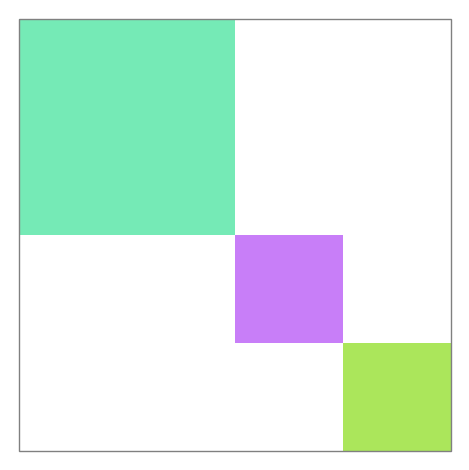}
  \caption{The figure shows the planted partition $\cP$ (left) and a `decoy' partition $\cD$ (right) for the SBM which has a planted partition of four equal-size parts. The partitions are denoted by showing the adjacency matrix of $G$ and colouring cell $u,v$ with colour $i$ if nodes $u,v$ are placed together in part / community $i$ in the partition. A `decoy' partition $\mathcal{D}$ is the one in which two of the planted blocks are clustered together in the same part, and each of the remaining planted blocks form their own part.}\label{fig.planted_and_decoy}
\end{figure}

\paragraph*{Intuition.} Consider a `decoy' partition $\mathcal{D}$ in which two of the planted blocks are clustered together in the same part, and each of the remaining planted blocks forms its own part. (There are $\binom{k}{2}$ such decoy partitions by symmetry.) We call $\cD$ a `decoy' partition since it has a `misleadingly' high modularity score given its distance to the planted partition $\cP$; see Figure~\ref{fig.planted_and_decoy}. 
(We emphasise that the decoy $\cD$ is an alternate partition of the nodes rather than an alternate random graph model to generate the graph.)

Given the decoy partition $\cD$, the parameters for OGP are set as follows.
We set $\mu = \q(G)-q_\mathcal{D}(G)+\eta$ for some small $\eta>0$. The idea is that partitions $\cB$ with $q_\cB(G)\geq \q(G) - \mu$ will include the decoy partition, those \emph{very} `close to' the decoy, as well as partitions somewhat `close to' the planted (and modularity optimal) partition.
Then we set $\nu_2 = d(\cP, \mathcal{D})-\delta = 1/k - \delta$ for some small $\delta>0$ so that the decoy partition and those very close to it will fall within the second interval. Lastly, we need to set $\nu_1$. Notice that, for small distances from the planted partition, the modularity score of partitions decreases with distance, so we set $\nu_1$ to be some distance by which the likely modularity score has dipped below $\q(G)-\mu$. This explains the intuition behind Theorem~\ref{thm:OGP}.

As an implication of the OGP we show that a natural algorithm based on Markov Chain 
Monte Carlo (MCMC) updates takes exponential time to reach any proximity of the 
ground truth. We state our result informally, and defer the formal statement as well as the proof to Appendix~\ref{subsection:slow-mixing}.
\begin{theorem}[Informal]\label{thm:slow_mixing_informal}
For every $\zeta>0$, there exists a large enough inverse temperature parameter
$\beta>0$, such that when the chain is initiated at distance
at least $1/k$ from the ground truth, the time to reach a partition close
to the ground truth is at least $\exp(\Theta(n))$ with high probability.
\end{theorem}

\subsection{Modularity as a predictor of the planted partition}\label{subsec:close_is_close}

In a breakthrough result of Bickel and Chen~\cite{bickel2009nonparametric}, it was shown that in the regime of growing degree, the modularity optimal partition of the stochastic block model is within distance~$o(1)$ of the planted partition (see also~\cite{bickel2015correction}): 

\begin{theorem}[\cite{bickel2009nonparametric,bickel2015correction}]
\label{thm.SBMrecovery}
Fix real numbers $a>b>0$, integer $k\geq 2$ and let $\omega=\omega(n) \to \infty$ as $n \to \infty$. 
Let $p=\omega\, a/n$ and $q=\omega\, b /n$, and let $G \in \mathcal{G}(n,k,p,q)$ with the planted partition~$\cP$. Then,
\[
d\big( \arg \max q_\cA(G) , \, \cP \big) = o(1).
\]
\end{theorem}

Since modularity-based clustering algorithms remain the most popular, despite known problems such as the resolution limit \cite{fortunato2007resolution}, it is reassuring that the modularity optimal partition is quite close to the `correct' partition in the SBM, a natural simplified model of community structure.  

However, we would like to know more. In particular, what about partitions with nearly optimal modularity score? Are these partitions also close to the planted partition? Our Theorem~\ref{thm.close_is_close} (along with some bounds on modularity scores for partitions at higher distances) implies that the answer is `yes': if the partition has a score within a certain distance of the maximal modularity value, then the partition must be near the planted one; see Appendix~\ref{app:close_is_close} for the proof.

\begin{theorem}    
[\bf Partitions with near-optimal modularity are close to the planted]\label{thm.close_is_close} 
Fix $a>b>0$ and integer $k\ge 2$. Let $p=p(n)=\omega\; a / n$ and $q=q(n)=\omega \; b/n$. For any $\eps>0$ and $\delta<\frac{a-b}{a+(k-1)b} \cdot \frac{2}{k^2}$, there exists $c=c(\eps)$ such that, if $\omega>c$ for $n$ large enough, then the following holds with probability at least $1-\eps$ provided $n$ is large enough. 

Let $G \in \mathcal{G}(n,k,p,q)$ with the planted partition $\cP$. 
Then, $q_\cA(G)>\q(G) - \delta $ for some $k$-part partition $\cA$ implies that $d(\cA, \cP) < \delta' + \varepsilon$,
where $\delta' =\delta'(\delta) \in \big(0,\frac{1}{k(k-1)}\big)$.
\end{theorem}

\subsection{Modularity at distance at most $1/k$}

The next theorem relates the maximum modularity of a partition to its distance from the planted partition, and is the main ingredient in the proof of Theorem~\ref{thm:OGP}.
For any $k \ge 2$, we approximate the maximum modularity value over all partitions of $k\ge 2$ parts $\cA$ at distance~$t/k$ (for some $t \in [0,1]$) from the planted partition $\cP$.

\begin{theorem}[\bf Maximal modularity at distance at most $1/k$]\label{thm:main_theorem} 
Fix $a>b>0$. Let $p=p(n)=\omega\; a / n$ and $q=q(n)=\omega \; b/n$. For any $\eps>0$, there exists $c=c(\eps)$ such that if $\omega>c$ then the following holds with probability at least $1-\eps$ provided $n$ is large enough.

Let $G \in \mathcal{G}(n,k,p,q)$ with the planted partition $\cP$. 
For any $d \in \left[0,\frac{1}{k}\right]$, 
let 
$$H(d) = \max_{\cA: |\cA|\le k} \left\{ q_\cA(G) : d - \frac{1}{\sqrt{n}} \le d(\cA, \cP) \le d + \frac{1}{\sqrt{n}} \right\}.
$$
Then,
\begin{equation*}
\left| H(d)  - \frac {a-b}{a+(k-1)b}\cdot h(d) \right| \le \eps,
\end{equation*}
where 
\[h(d)= 1 - \frac {1}{k} - 2d\big (1-d(k-1)\big). \]
\end{theorem}
The random variable $H(d)$ is defined as the maximal modularity over all $k$-part partitions with the distance from $\cP$ in an interval $d\pm 1/\sqrt{n}$. We keep the margins $\pm 1/\sqrt{n}$ for the technical reason that there might not be any partition at distance exactly $d$ (say, when $d = \pi / 4k$, since the distance is always a rational number). Theorem~\ref{thm:main_theorem} establishes that $H(d)$ is likely to be well approximated by the deterministic expression $\frac{a-b}{a+(k-1)b}\cdot h(d)$.

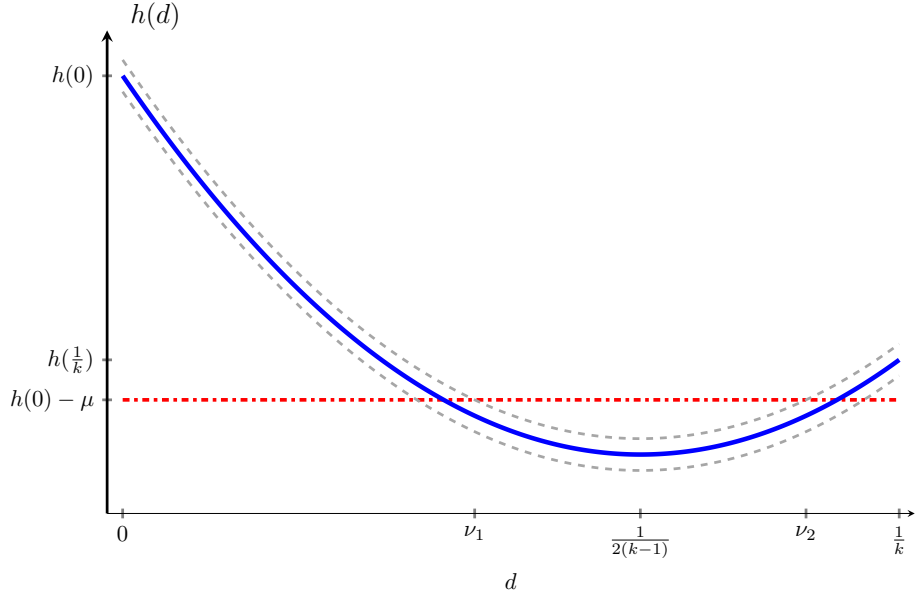
\begin{figure}
\begin{center}
\hspace{-5mm}
\begin{tikzpicture}[scale=0.86]

\def\k{4}          
\def\tau{0.22}     

\def\eps{0.007} 

\pgfmathsetmacro{\xmax}{1/\k}
\pgfmathsetmacro{\xmid}{1/(2*(\k-1))}

\pgfmathsetmacro{\tauone}{1/(\k-1) - \tau}
\pgfmathsetmacro{\tautwo}{\tau}

\pgfmathsetmacro{\hzero}{1 - 1/\k}
\pgfmathsetmacro{\honek}{1 - 1/\k - 2/(\k*\k)}
\pgfmathsetmacro{\hnu}{1 - 1/\k - 2*\tau*(1 - \tau*(\k-1))+\eps}

\pgfmathsetmacro{\muval}{\hzero - \hnu}

\pgfmathdeclarefunction{h}{1}{%
  \pgfmathparse{1 - 1/\k - 2*(#1)*(1 - (#1)*(\k-1))}%
}

\begin{axis}[
  width=14cm,
  height=9cm,
  axis lines=left,
  xmin=-0.02*\xmax, xmax=1.02*\xmax,
  ymin=\hnu-0.05, ymax=\hzero+0.02,
  y axis line style={line width=1.2pt}, 
  tick style={line width=1.2pt},
  clip=false,
  scaled x ticks=false,
  xtick={0,\tauone,\xmid,\tautwo,\xmax},
  xticklabels={$0$, $\nu_1$,$\frac{1}{2(k-1)}$,$\nu_2$,$\frac{1}{k}$},
  ytick={\hzero,\honek,\hnu},
  yticklabels={$h(0)$,$h(\frac{1}{k})$,$h(0)-\mu$},
  xlabel={$d$},
  ylabel={},
]

\addplot[
  red,
  line width=1.6pt,
  dash dot
] coordinates {(0,\hnu) (\xmax,\hnu)};

\addplot[
  gray!70,
  line width=1.2pt,
  dashed,
  domain=0:\xmax,
  samples=400
] {h(x) + \eps};

\addplot[
  gray!70,
  line width=1.2pt,
  dashed,
  domain=0:\xmax,
  samples=400
] {h(x) - \eps};

\addplot[
  blue,
  line width=2.0pt,
  domain=0:\xmax,
  samples=400
] {h(x)};

\node[font=\Large] at (axis description cs:0.06,1.03) {$h(d)$};

\end{axis}
\end{tikzpicture}
\caption{A graph of the function $h(d)$ for $k=4$ and the relation to the OGP parameters $\mu, \nu_1$ and $\nu_2$ -- see the text following Theorem~\ref{thm:main_theorem}. The dashed grey lines indicate an $\eps$-region above and below~$h(d)$. By Theorem~\ref{thm:main_theorem}, any partition $\cA$ at distance $d$ is likely to have modularity at most $h(d)+\eps$. From this, we will be able to conclude that it will be likely that any partition $\cA$ with modularity at least $h(0)-\mu$ (i.e., above the red line)  will be at a distance $d\leq \nu_1$ or $d\geq \nu_2$, thus establishing the OGP with these parameters.}\label{fig:functions_h_new}
\end{center}
\end{figure}

Note that function $d\mapsto h(d)$ decreases on $\left[0, \frac {1}{2(k-1)}\right]$ and increases on $\left[\frac {1}{2(k-1)},\frac{1}{k}\right]$. As a result, 
\begin{itemize}
\item $h(0) = 
1 - \frac {1}{k}$, 
\item $%
h(d)$  
is decreasing on $[0, \frac {1}{2(k-1)}]$ reaching $h\left(\frac {1}{2(k-1)}\right) = 
1 - \frac {1}{k} - \frac {1}{2(k-1)}$,
\item $%
h(d)$ is increasing on $[\frac {1}{2(k-1)},\frac{1}{k}]$ reaching $h(1/k) = 
1 - \frac {1}{k} - \frac {2}{k^2}$.
\end{itemize}

One can easily separate $h(1/k)$ from its local minimum $h(\frac {1}{2(k-1)})$ by introducing a threshold $\mu$ in Theorem~\ref{thm:OGP}, say, $\mu = 
h(0)-h(\nu)$ with $\nu\in \left(\frac{1}{2(k-1)}, \frac{1}{k}\right)$. In Figure~\ref{fig:functions_h_new}, we illustrate $h(d)$ for $d\in[0,\frac{1}{k}]$ and the interplay with the OGP parameters $\mu, \nu_1$ and $\nu_2$.   %

Below, we will show Theorem~\ref{thm:OGP}, i.e.,\ that modularity has OGP in the SBM -- assuming Theorem~\ref{thm:main_theorem}, where we recall that Theorem~\ref{thm:main_theorem} characterises the maximum modularity score over partitions $\cA$ at given distances $d\in [0,\frac{1}{k}]$ from the planted partition~$\cP$. We outline the proof of Theorem~\ref{thm:main_theorem} in Sections~\ref{sec:distance}--\ref{sec:optimisation}, and give full details in Appendices~\ref{app:remainder_of_proof} and~\ref{app:optimisation}. We prove Theorem~\ref{thm:OGP} using Theorem~\ref{thm:main_theorem} in Appendix~\ref{app:proof-OGP}.

The asymptotic maximal modularity value in $\mathcal{G}(n,k,p,q)$ is known and also known to coincide with the modularity of the planted partition; we state it in Theorem~\ref{thm:modularity_function} below for completeness. The result follows from Theorem~\ref{thm.SBMrecovery} and was also given in~\cite{sampling} for slightly higher $p,q$. We note that (\ref{eq:qP}) in Theorem~\ref{thm:modularity_function} is a direct corollary of Theorem~\ref{thm:main_theorem} as a special case when~$d=0$. Also~$(5)$ follows, for example, by~Lemma~\ref{lem:write_mod_as_signature}.

\begin{theorem}[\bf Maximal modularity]\label{thm:modularity_function} 
Fix $a>b>0$ and integer $k\geq 2$. Let $p=p(n)=\omega\; a / n$ and $q=q(n)=\omega \; b/n$. For any $\eps>0$, there exists $c=c(\eps)$ such that, if $\omega>c$ for $n$ large enough, then the following holds with probability at least $1-\eps$ provided $n$ is large enough.
Let $G \in \mathcal{G}(n,k,p,q)$ and denote the planted partition by $\cP$. Then,
\begin{equation}\label{eq:qP}
\left| q_\cP(G)  - \frac {a-b}{a+(k-1)b} \left( 1 - \frac {1}{k} \right) \right| \le \frac {\eps}{2}
\end{equation}
and 
\begin{equation}\label{eq:q*}
\Big| q^{\star}(G) - q_\cP(G) \Big| \le \frac {\eps}{2}.
\end{equation}
As a result, 
$$
\left| q^{\star}(G)  - \frac {a-b}{a+(k-1)b} \left( 1 - \frac {1}{k} \right) \right| \le \eps.
$$
\end{theorem}
Note that, as expected, $\q(G)\approx \frac {a-b}{a+(k-1)b} 
 \,h(0)$, where $h(\cdot)$ is defined in Theorem~\ref{thm:main_theorem}.

\section{Ingredients of the proof 
: Distances to the planted partition $\cP$}\label{sec:distance}

For now, for simplicity of exposition, we assume that $n$ is divisible by~$k$, as this will be assumed in the first lemma, Lemma~\ref{lem:dist_as_dplus}. However, in the proof of Theorem~\ref{thm:main_theorem}, we relax this condition and deal with the general case. 

To prove our main result, we need to investigate a family of partitions of $[n]$ with a given distance to the planted partition $\cP = (P_i)_{i=1}^k$. It will be convenient to represent partitions as $k \times k$ matrices that capture the way how these partitions overlap with $\cP$. Let $\cA = (A_i)_{i=1}^k$ be any partition of $[n]$ into $k$ parts. Then, the \emph{signature} of $\cA$ (with respect to $\cP$) is defined as the matrix $X = X(\cA) =(x_{ij})_{1\leq i\leq j \leq k}$, where 
\begin{equation}\label{eq:signature}
x_{ij} = \frac {|A_i \cap P_j|}{n/k}.
\end{equation}
Note that for the signature to be well defined, one needs to fix the labelling of both $\cP$ and $\cA$. To compute the distance between $\cP$ and $\cA$, one needs to align these labels as best as possible (see Lemma~\ref{lem:dist_as_dplus} below), but the modularity score clearly does not depend on the way the labels are aligned (see Lemma~\ref{lem:write_mod_as_signature} below).
Note also that, trivially, $x_{ij} \ge 0$ for all $i,j \in [k]$. Moreover, since $A_1 \cup \cdots \cup A_k = [n]$, for any $j \in [k]$,
\begin{equation}\label{eq:sum1}
\sum_{i \in [k]} x_{ij}  = \frac {1}{n/k} \sum_{i \in [k]} |A_i \cap P_j| = \frac {|P_j|}{n/k} = 1.
\end{equation}
Finally, note that the signature of the planted partition $\cP$ is the matrix $Y^{\cP}=(y_{ij})_{1\leq i \leq j \leq k}$ with $y_{ii}=1$ for all $i \in [k]$ and $y_{ij}=0$ for all $i\neq j$.

In our first lemma, we rewrite the distance between any partition $\cA$ and the planted partition $\cP$ in a convenient form, namely, as a function of the signature of $\cA$. The proof is short so we give it below. 

\begin{lemma}\label{lem:dist_as_dplus}
Suppose that $n$ is divisible by $k$. 
Let $\cA$ be any partition of $[n]$ with $k$ parts, and let $\cP$ be the planted partition.
Let $X=X(\cA)=(x_{ij})_{1\leq i \leq j \leq k}$ be the signature of $\cA$. Then,
\begin{equation}\label{eq:dist_as_dplus}
d(\cA, \cP) = 1-\tfrac{1}{k}\max_{\sigma} \sum_{i=1}^k x_{\sigma(i) i}, 
\end{equation}
where the maximum is taken over all permutations of $[k]$. 

Moreover, $d(\cA, \cP) \leq 1 -1/k$ and this upper bound is sharp, that is, there exists a partition~$\cA$ such that $d(\cA, \cP) = 1 -1/k$, provided that $n$ is divisible by $k^2$. 
\end{lemma}

\begin{proof}[Proof of Lemma~\ref{lem:dist_as_dplus}.]
Equality~(\ref{eq:dist_as_dplus}) follows immediately from the definition of the signature in \eqref{eq:signature} and the definition~(\ref{eq:dist_meas}) of the distance between the two partitions. Indeed, 
$$
d(\cA, \cP) = 1 - \frac{1}{n} \max_\sigma \sum_{i=1}^k |A_{\sigma(i)} \cap P_i|  
= 1 - \frac{1}{n} \max_\sigma \sum_{i=1}^k x_{\sigma(i)i} \cdot (n/k)
= 1 - \frac{1}{k} \max_\sigma \sum_{i=1}^k x_{\sigma(i)i}.
$$
It remains to show that $d(\cA, \cP) \le 1 - 1/k$ or, equivalently, that $\max_\sigma \sum_{i=1}^k x_{\sigma(i)i} \ge 1$.

Consider the sum $S = \sum_\sigma \sum_{i=1}^k x_{\sigma(i)i}$, where the outer sum is taken over all $k!$ permutations of $[k]$. Clearly, there are $k! \cdot k$ terms in $S$ and for each $a,b \in [k]$, the term $x_{ab}$ occurs exactly $(k-1)!$ times ($i$~has to be equal to $b$ and there are $(k-1)!$ permutations that map $b$ to $a$). Hence, for each $j \in [k]$, the sum $\sum_{i=1}^k x_{ij}$ (which is equal to 1 by~(\ref{eq:sum1})) occurs $(k-1)!$ times. We conclude that $S = k \cdot (k-1)! = k!$. By an averaging argument, there exists a permutation $\hat{\sigma}$ for which $\sum_{i=1}^k x_{\hat{\sigma}(i)i}\ge 1$. Hence, $\max_\sigma \sum_{i=1}^k x_{\sigma(i)i} \ge \sum_{i=1}^k x_{\hat{\sigma}(i)i} \ge 1$, thus the desired inequality holds.

Assume now that $n$ is not only divisible by $k$ but, in fact, it is divisible by $k^2$. We construct a partition $\cA$ by partitioning each $P_j$ into $k$ equal parts and then picking $1/k$ fraction of each $P_j$ to form $A_i$ of size $n/k$. This partition has signature $X$ with $x_{ij} = 1/k$ for all $i,j \in [k]$. We get that $\max_\sigma \sum_{i=1}^k x_{\sigma(i)i} = \sum_{i=1}^k 1/k = 1$, which shows that the upper bound is sharp. This finishes the proof of the lemma.
\end{proof}

\section{Ingredients of the proof: Concentration}
\label{sec:concentration}

The signature of $\cA$ not only determines the distance between $\cA$ and $\cP$ but, more importantly, it predicts (up to an arbitrarily small error in  a sufficiently dense graph $G$) the modularity score $q_\cA(G)$ of $\cA$ for $G \in \mathcal{G}(n,k,p,q)$ generated by the SBM. Indeed we will prove that the modularity score of $\cA$ is well-approximated by a scaling of $g(X)$, a function only of the signature $X$ of $\cA$ that is defined by
\begin{equation}\label{defn.gx}
g(X) = \sum_{i=1}^k \sum_{1 \le j < j' \le k} (x_{ij} - x_{ij'})^2.
\end{equation}

\begin{lemma}
[\bf Concentration of modularity]\label{lem:write_mod_as_signature}
Fix $a>b>0$ and integer $k\geq 2$. Let $p=p(n)=\omega\; a / n$ and $q=q(n)=\omega \; b/n$. For any $\eps>0$, there exists $c=c(\eps)$ such that if $\omega>c$ for $n$ large enough, then the following holds with probability at least $1-\eps$ provided $n$ is large enough.

Let $G \in \mathcal{G}(n,k,p,q)$. For all $k$-part partitions $\cA$, 
\[ 
\left| q_\cA(G) - \frac{a-b}{(a+(k-1)b)} \cdot \frac {g(X)}{k^2} \right| < \eps, 
\]
where $X = X(\cA) =(x_{ij})_{1\leq i,j \leq k}$ is the signature of $\cA$ and $g(X)$ is defined  in~\eqref{defn.gx}.
\end{lemma}

Recall that the planted partition $\cP$ has signature $Y^{\cP}=(y_{ij})_{1\leq i \leq j \leq k}$ with $y_{ii}=1$ for all $i \in [k]$ and $y_{ij}=0$ for all $i\neq j$. Since $g(Y^{\cP}) = k \cdot (k-1)$, we get that $q_\cP (G)$ is very close to $\left(1-\frac {1}{k}\right) \cdot \frac{a-b}{a+(k-1)b}$.

As one might want to use this result for other purposes, we state it (and prove it, of course) without assuming that $n$ is divisible by $k$. The full proof of Lemma~\ref{lem:write_mod_as_signature} can be found in {Appendix~\ref{subsec:write_mod_as_signature},} though we describe some intermediate lemmas below.

One of these intermediate lemmas to prove Lemma~\ref{lem:write_mod_as_signature}, namely Lemma~\ref{lem:weightedSBMk},  concerns the modularity score of partitions $\cA$ with given signature $X$ on a deterministic \emph{weighted graph}. Lemma~\ref{lem:weightedSBMk} can be regarded as a mean-field calculation. Indeed, working with the deterministic weighted graph means we may consider, for example, the edge-weight of edges in part $A_i$ between the planted blocks $P_j$ and $P_{j'}$, which corresponds to the \emph{expected} number of such edges. 
After proving Lemma~\ref{lem:weightedSBMk}, to prove the concentration result Lemma~\ref{lem:write_mod_as_signature}, it remains only to make a link between the modularity scores of our random graph and those of the deterministic weighted graph. We provide both of these proofs in Appendix~\ref{app:remainder_of_proof}.

\section{Ingredients of the proof: Optimisation}
\label{sec:optimisation}

We have now established that, up to small errors, we may consider $g(X)$ as our objective function, and have re-written our distance to the planted partition in terms of $X$. Thus, the new aim is to maximise $g(X)$ given the distance of $X$ to the planted partition $Y^{\mathcal{P}}$.

\begin{lemma}\label{lem:optimisation} Let $0\leq t <1$.
Define $\mathcal{X}_k(t)$ to be the family of matrices $X=(x_{ij})_{1\leq i \leq j \leq k}$ such that
\begin{align}
\nonumber &x_{ij} \geq 0, &\mbox{for all $i,j \in [k]$},\\[0.5em]
\label{eq.def.X.orig} &\sum_i x_{ij}  =  1 &\mbox{for all $j \in [k]$},  \\
\nonumber &\sum_{i\neq j}x_{ij}  = t, \mbox{ and }\\
\nonumber &\sum_{i}x_{ii} \geq \max_{\sigma} \sum_i x_{\sigma(i)i}.
\end{align}
Let $g(X)$ be defined as in~\eqref{defn.gx}. Then,
\[ \max_{X\in \mathcal{X}_k(t)} g(X) = k(k\!-\!1) - 2t \big( k\!-\!t(k\!-\!1)\big)\;\; \mbox{and} \;\; \arg\max_{X\in \mathcal{X}_k(t)} g(X) = \{ X_k^{(ij)}(t) \; : \; i\neq j \}, \] 
where, for $i\neq j$, $X^{(ij)}_k(t)$ is the $k\times k$ matrix with $x_{ij}=t$, $x_{jj}=1-t$, all other diagonal elements $1$ and non-diagonal elements~$0$ (i.e., $x_{aa}=1$ for all $a\neq j$ and $x_{ab}=0$ for all $a\neq b$ such that $(a,b)\neq (i,j)$).
\end{lemma}

Note that in the optimiser $X_k^{(i,j)}(t)$, the row sums are $(1+t, 1-t, 1, \ldots, 1)$, corresponding to part sizes $(1/k+d, 1/k-d, 1/k, \ldots, 1/k )\,\cdot n$ in the partition at distance $d<1/k$ with $d=t/k$ (or, thus, equivalently, $t<1$).

The following lemma shows that the maximal $g(X)$ for signatures at distance $d>1/k$ (i.e., for $t>1$) is at most the maximal $g(X)$ for signatures at distance $d=1/k$. This result is important to several of our proofs as it will allow us to infer that partitions $\cA$ with high modularity scores must be at distance $d\leq 1/k$ from the planted partition.

\begin{lemma}\label{lem:opt_far_mod} 
Consider $\mathcal{X}_k(t)$ as defined in~\eqref{eq.def.X.orig}, and $g(X)$ as defined in~\eqref{defn.gx}.
Then, for any $t>1$,
\[ \max_{X\in \mathcal{X}_k(t)} g(X) < k(k-1)-2.\] 
\end{lemma}

The optimisation problems in Lemmas~\ref{lem:optimisation} and~\ref{lem:opt_far_mod} are quadratic maximisation problems. (Note that in general quadratic maximisation problems are NP-hard~\cite{sahni1974computationally}.) We prove both lemmas in Appendix~\ref{app:optimisation}. Lemma~\ref{lem:optimisation} is the last ingredient needed to prove Theorem~\ref{thm:main_theorem}, which gives the likely optimal modularity score over all partitions at distance $d\in[0,1/k]$ from the planted partition. (Recall that we will prove our main OGP result, Theorem~\ref{thm:OGP}, in Appendix~\ref{app:proof-OGP} by appealing to Theorem~\ref{thm:main_theorem}.) {\Fc Lemma~\ref{lem:opt_far_mod} is used to show slow mixing of a natural MCMC in Theorem~\ref{thm:slow_mixing_formal}.}

To understand how these lemmas fit in, recall that the distance to the planted partition $\cP$ of a partition $\cA$ with signature $X$ is $d(\cA,\cP)=1-\frac{1}{k}\max_{\sigma}\sum_{i} x_{\sigma(i)i}=\frac{1}{k}\min_{\sigma} \sum_{i\neq j} x_{\sigma(i)j}$, by~Lemma~\ref{lem:dist_as_dplus} and since $\sum_j x_{ij}=1$ for each $i$, so that $x_{\sigma(i)i}=1-\sum_{j\colon j\neq i}x_{\sigma(i)j}$. The condition $\sum_{i\neq j} x_{ij} = t$ corresponds to enforcing a distance of $t/k$ from the planted partition, and the $\arg \max$ tells us the set of partitions that achieve the maximal modularity score at distance $t/k$. Note that, for $t=1$, i.e.\ at distance of $d=1/k$ from the planted partition, $X^{(ij)}_k(1)$ is the signature for the `decoy' partition $\cD$ where two planted blocks $P_i$ and $P_j$ are placed within the same part, and all other planted blocks are placed within their own part -- see also Figure~\ref{fig.planted_and_decoy}.

\section{Discussion and the case of balanced partitions.
}\label{sec.conj}
The literature contains positive results, with algorithms using local updates based on the modularity function to recover communities in the SBM~\cite{cohen2020power,cohen2022massively}. This paper proves OGP, a~signature of algorithmic hardness, for such algorithms. 

Of course this does not give rise to a contradiction, since the setups in~\cite{cohen2020power,cohen2022massively} and ours are subtly different and also OGP has been exhibited for problems known to be easy~\cite{li2024some}. However, it leads to interesting open questions, probing which of the differences in the two setups are important. We describe this now in more detail.

As mentioned earlier, for $k=2$ communities and starting with a random partition into equal-size parts, a local algorithm based on the modularity function was shown to recover the planted communities~\cite{cohen2020power}. The local moves for this algorithm take the form of a \emph{swap}, which takes pairs of vertices, one from each part, and swaps them if this increases the modularity score. This naturally maintains equal-size parts.  
For general $k$, an algorithm with parallel local updates based on the modularity function was shown to recover the ground truth partition~\cite{cohen2022massively}. We note that this parallel algorithm has a random balanced start and also a mechanism to maintain the balanced sizes of parts during the algorithm.

For the overlap gap property that we establish, it was important that the `decoy' partition~$\cD$, which is an {\em unbalanced} partition into $k-1$ parts, has a `surprisingly high' modularity given its distance of $1/k$ from the ground truth partition $\cP$. Furthermore, at distances $0<d<1/k$ from $\cP$ modularity, optimal partitions are increasingly unbalanced, with part sizes $(1/k+d, \; 1/k-d, \; 1/k, \ldots, 1/k )$; see Lemma~\ref{lem:optimisation}.

Perhaps a crucial difference in the two setups is that the positive results of~\cite{cohen2020power, cohen2022massively} had {\em balanced} partitions, both in the initialisation as well as during the algorithm. We finally show that if we introduced this balanced condition into our optimisation problem then it no longer exhibits OGP. In particular, when considering $\max_{\cA} q_{\cA}$ where the maximisation is over \emph{balanced} $k$-part partitions instead of all $k$-part partitions, we see that this no longer exhibits OGP.  
We should note that this by itself does not imply  
{\Fc success of particular greedy algorithms -- and hence fast algorithms.}
More work needs to be done to exhibit one {(as done in~\cite{cohen2020power,cohen2022massively})}. In a similar vain, for the problem of sparse regression~\cite{gamarnik2022sparse}, in the regime where OGP ceases to hold a greedy type algorithm was established to be effective. The construction is not based on OGP. Instead it relies directly on the properties of the model, which is typically the case.

To state this result, we fix the random graph $G \in \mathcal{G}(n,k,p,q)$ under the conditions of Theorem~\ref{thm:main_theorem}. It turns out that the maximum modularity of $G$ only over \emph{balanced} partitions at distance about $0<d<1-1/k$ from the planted partition $\cP$ concentrates about some function that now {\em is} decreasing in~$d$:

\begin{restatable}
{proposition}{propbalanced}\label{prop.balanced}
Let $\mathcal{X}_t'$ be defined as in~\eqref{eq.def.X.orig} with the additional restriction that $\sum_j x_{ij} =1$ for all $i\in[k]$. Then $g^{\rm bal}(t)=\max_{X \in \mathcal{X}_t'}g(X)$ is {\em strictly} decreasing on $0<t<k-1$.
\end{restatable}

Proposition~\ref{prop.balanced} implies an analogue of~Theorem~\ref{thm:main_theorem}, optimising now only over \emph{balanced partitions}, for $0<d<1-1/k$, now with decreasing $\tilde{h}(d),$ contrary to 
$h(d)$. %

\bibliography{articles}

\newpage
\appendix
\section{A guide to the proofs given in the appendix}
\subsection{Main results: OGP for SBM and slow-mixing}

The main results in the paper are that (a) for three or more communities, the stochastic block model (SBM) exhibits the overlap gap property (OGP); and (b) failure of a natural MCMC algorithm. 

For (a), that SBM exhibits OGP, the statement is  Theorem~\ref{thm:OGP}, and is proven in Appendix~\ref{app:proof-OGP}. Proving Theorem~\ref{thm:OGP} requires understanding of the maximum modularity of partitions within distance $d\leq 1/k$ of the planted partition (Theorem~\ref{thm:main_theorem}). Indeed, proving Theorem~\ref{thm:main_theorem} is the main work of the paper, and we outline the steps in Appendix~\ref{subsec:steps_to_prove_max_mod}.  

For (b), the failure of a natural Markov Chain Monte Carlo (MCMC) algorithm, the statement is~Theorem~\ref{thm:slow_mixing_formal} (an informal version, Theorem~\ref{thm:slow_mixing_informal}, appeared in the body of the paper).
This algorithm is described in Appendix~\ref{sec:slow-mixing}, and we show in Proposition~\ref{prop:Planted-dominates} that it will likely output a partition within distance $o(1)$ of the planted partition~$\cP$. However, the main result is a negative one: Theorem~\ref{thm:slow_mixing_formal} shows that the time taken is exponentially large in $n$. In Appendix~\ref{sec:slow-mixing}, we state and prove both Theorem~\ref{thm:slow_mixing_formal} and Proposition~\ref{prop:Planted-dominates}, assuming the OGP result stated in Theorem~\ref{thm:OGP} and a bound on the maximum modularity of partitions at distance $d>1/k$ of the planted partition (Lemma~\ref{lem:opt_far_mod}).

\subsection{Putting it all together: Proving Theorem~\ref{thm:main_theorem} (which proves Theorem~\ref{thm:OGP})}\label{subsec:steps_to_prove_max_mod}

We outline the proof of Theorem~\ref{thm:main_theorem}; in terms of the three ingredients discussed in the main body of the paper: distances to the planted partition (Section~\ref{sec:distance}), concentration (Section~\ref{sec:concentration}) and optimisation (Section~\ref{sec:optimisation}).  

An important definition was that of the signature $X$ of a partition $\cA$; see \eqref{eq:signature}.
The signature of $\cA$ determines the distance between $\cA$ and $\cP$. This was established in Section~\ref{sec:distance}, with all proofs contained within that section. The second and third ingredients, namely, concentration and optimisation, were more involved with results left to be proven in the appendix; see Appendices~\ref{subsec:describe_concentration} and~\ref{sec:figures_and_calcs_for_figures}. The corresponding proofs are in Appendices~\ref{app:concentration} and~\ref{app:optimisation}, respectively.

\subsection{Ingredient: Concentration}\label{subsec:describe_concentration}

We just saw that the distance of partition $\cA$ to the planted partition is determined by the signature of $X$.
Moreover, the signature (up to a small error) defines the modularity score $q_\cA(G)$ of $\cA$ for $G \in \mathcal{G}(n,k,p,q)$ generated by the SBM. This is established by the concentration results in Section~\ref{sec:concentration} which show that for $\cA$ with signature $X$, the modularity score $q_\cA(G)$ is very close to $\frac{a-b}{a+(k-1)b}\,\frac{1}{k^2}\,g(X)$. See~\eqref{defn.gx} for a definition of $g(X)$. 
The proofs for the concentration results may be found in Appendix~\ref{app:concentration}.

\subsection{Ingredient: Optimisation}\label{sec:figures_and_calcs_for_figures}

We are interested in the maximum modularity at a given distance from the planted partition, and the concentration results tell us, loosely, that the modularity score $q_\cA(G)$ is quite close to $\frac{a-b}{a+(k-1)b}\,\frac{1}{k^2}\,g(X)$, where $X$ is the signature of $\cA$. Hence, the remaining challenge is to understand the maximal value of $g(X)$ over signatures $X$ of partitions at a given distance to the planted partition $\cP$. 

We recall the setup for our optimisation problem. We have defined $\mathcal{X}_k(t)$ to be the family of matrices $X=(x_{ij})_{1\leq i \leq j \leq k}$ such that~\eqref{eq.def.X.orig} holds.
Recall from~\eqref{defn.gx} that we have defined $g(X)$ to be 
$g(X) = \sum_{i=1}^k \sum_{1 \le j  < j' \le k}(x_{ij} - x_{ij'})^2$.

The two main optimisation results are to (a) understand $\max_{X\in \mathcal{X}_{k}(t)}g(X)$ for $t\in [0,1]$ (corresponding to $d\leq 1/k$); and (b) upper bound $\max_{X\in \mathcal{X}_{k}(t)}g(X)$ for $t>1$ (corresponding to $d> 1/k$). For part (a), the result is stated in Lemma~\ref{lem:optimisation}, and the proof is in Appendix~\ref{app:optimisation-lemma}. For part (b), the result is stated in Lemma~\ref{lem:opt_far_mod}, and the proof is in Appendix~\ref{subsec:far_mod}.

\subsection{Auxiliary result : Theorem~\ref{thm.close_is_close}}
A result of Bickel and Chen~\cite{bickel2009nonparametric} states that the modularity-optimal partition is likely to be within small distance of the planted partition. We extend this to say that partitions with modularity score very close to optimal are within a small distance of the planted partition. See Theorem~\ref{thm.close_is_close} for details, and Appendix~\ref{app:close_is_close} for proofs.

\subsection{Outline}
We briefly outline the following sections. Appendix~\ref{sec:slow-mixing} proves the failure of MCMC algorithms, Appendix~\ref{app:concentration} proves the concentration results and Appendix~\ref{app:optimisation} proves the optimisation results. After these sections, we prove our main result, that modularity in SBM has OGP, in Appendix~\ref{proof.OGP}. In Appendix~\ref{app:close_is_close}, we prove the  auxiliary result Theorem~\ref{thm.close_is_close}.
We close in Appendix~\ref{app:balanced} by studying balanced partitions and proving Proposition~\ref{prop.balanced}.

\section{Failure of the Markov Chain Monte Carlo algorithm}\label{subsection:slow-mixing}\label{sec:slow-mixing}
Throughout this section we consider $k$-partitions with $k$ fixed.  We denote $q_{\cA}=q_{\cA}(G)$, $q^{\star}=q^{\star}(G)$ for brevity.
The OGP immediately implies the failure of a natural Greedy algorithm for finding the planted partition $\mathcal{P}$ when 
started from the decoy partition $\mathcal{D}$. More specifically, Greedy is an algorithm resulting in a sequence of partitions $\mathcal{A}_t$ built
as follows: We initialise the Greedy algorithm in the decoy partition, i.e., $\mathcal{A}_0=\mathcal{D}$. Given $\mathcal{A}_{t-1}$, the partition $\mathcal{A}_t$ is obtained 
by changing the membership of at most one node $u\in [n]$ such that the resulting modularity strictly increases. 
That is, $\cA_t$ is any partition satisfying
$d(\cA_t, \cA_{t-1})={ 1/n}$ and 
$q_{\mathcal{A}_t}(G)>q_{\mathcal{A}_{t-1}}(G)$. If no such node $u$ exists, then the algorithm stops and outputs the partition obtained
in the final step. By the OGP, the Greedy algorithm, initiated at the decoy partition $\cD$, terminates at this partition with distance at least $1/k$ from the ground truth.

A natural generalisation of the Greedy algorithm, which is guaranteed (as we will show) to output a partition approximately
matching the planted partition $\mathcal{P}$, is the well-known  Markov Chain Monte Carlo (MCMC) algorithm which we now describe. 
Our main result in this section is the proof of slow mixing of this MCMC. In particular, we will show that the time it takes for the algorithm to approximately produce an 
planted partition is exponentially large in $n$. 

We begin by describing the MCMC algorithm. A parameter usually called the {\em inverse temperature} $\beta$ is fixed. The algorithm proceeds as a Markov 
chain moving according to the following rules. Given any partition $\mathcal{A}$, consider any partition $\mathcal{A}'$ with 
$d(\mathcal{A},\mathcal{A}')={ 1/n}$. Then the algorithm moves from $\mathcal{A}$ to $\mathcal{A}'$ with probability proportional
to $\exp(\beta n q_{\mathcal{A}'})$, namely with probability
\begin{align*}
\pr\left(\mathcal{A}\to \mathcal{A}'\right) \triangleq \frac{ \exp(\beta n q_{\mathcal{A}'}) } {\sum_{\mathcal{B}} \exp(\beta n q_{\mathcal{B}})},
\end{align*}
where here the sum is over all partitions $\mathcal{B}$ with $d(\mathcal{A},\mathcal{B})={1/n}$. (Note that we choose this parametrisation so that the stationary distribution is informative.) 
It is known that the unique stationary 
distribution of this chain is the so-called \emph{Gibbs distribution} given by 
\begin{align*}
\pr_{\rm Gibbs}(\mathcal{A})=\frac{\exp(\beta n q_{\mathcal{A}})}{Z},
\end{align*}
where $Z=\sum_{\mathcal{B}} \exp(\beta n q_{\mathcal{B}})$ is the normalising constant, which is also called the {\em partition function}, and now the sum is over \emph{all} partitions $\mathcal{B}$.

We first show that this algorithm is sound, in the sense that for large enough $\beta$ it produces  partitions that are close to the ground truth $\mathcal{P}$.

\begin{proposition}\label{prop:Planted-dominates}
For every $\eps>0,\zeta>0$, there exists large enough $\beta>0$ such that 
\begin{align*}
\sum_{\mathcal{A}\colon  d(\mathcal{A},\mathcal{P}) \le \zeta} \pr_{\rm Gibbs}(\mathcal{A}) \ge 1-\eps.
\end{align*}
\end{proposition}
In words, the probability mass of partitions with distance at most $\zeta$ from the planted partition constitutes
at least a $1-\eps$ fraction of the total probability mass.
Thus, if the algorithm were to run until stationarity, then a partition sampled according to the Gibbs distribution is likely to be at most $\zeta$ close to the ground truth, modulo an at most $\eps$ likelihood event.

\begin{proof}[Proof of Proposition~\ref{prop:Planted-dominates}]
Given any  $\zeta>0$ let $\mathcal{E}(\zeta)=\lbrace \mathcal{A}\colon d(\mathcal{A},\mathcal{P})\le \zeta\rbrace$.
Fix $\eps,\zeta>0$. By Theorem~\ref{thm.close_is_close}, we can find $\delta>0$ small enough so that, with probability at least $1-\eps/2$,
the event occurs that every $\mathcal{A}$ satisfying $q_{\mathcal{A}}\ge q^{\star}-\delta$ satisfies $\mathcal{A}\in \mathcal{E}(\zeta)$.
On this event,
\begin{align*}
\sum_{\mathcal{A}\notin \mathcal{E}(\zeta)}\exp\left(\beta n q_{\mathcal{A}}\right)
&\le k^n\exp(\beta n (q^{\star}-\delta)).
\end{align*}
Fix $\beta$ large enough so that $\log k -\beta\delta \le -1$ (any strictly negative constant will do). Then, 
this sum is at most $\exp(-n)\exp(\beta n q^{\star})$, so that
\begin{align*}
\sum_{\mathcal{A}\notin \mathcal{E}(\zeta)}\pr_{\rm Gibbs}(\mathcal{A})
&\le
{\exp(-n)\exp(\beta n q^{\star}) \over Z} \\
& \le 
{\exp(-n)\exp(\beta n q^{\star}) \over \exp(\beta n q^{\star})}  \\
&=\exp(-n)\le \eps/2,
\end{align*}
for large enough $n$. Combining this event with the complement event, which occurs with probability at most $\eps/2$, we
complete the proof.
\end{proof}

Our main result, described next, shows that unfortunately the time to stationarity is exponentially large in $n$ when the chain
is initiated at a large enough distance from the partition, appropriately defined. 
Worse than that, the time to reach even \emph{one} partition close to  the ground
truth is exponentially large in $n$, again when the chain is initiated at a large distance from the ground truth. This means that the set of such starting points acts as a \emph{metastable set}. See \cite{BovHol15} for a discussion on metastability, and \cite{CojGalGolRavSteVig23} for an example of metastability in random regular graphs.

In preparation for the proof of this claim, we recap some of the properties implied by the OGP {\Fc and an auxiliary lemma}: there exist $0<\nu_1<\nu_2<1$ and $c_1,c_2>0$ such
that the following holds. Let 
\begin{align*}
\mathcal{E}_{\rm close} &= \lbrace \mathcal{A}\colon d(\mathcal{P},\mathcal{A})\le \nu_1 \rbrace,  \\
\mathcal{E}_{\rm far} &= \lbrace \mathcal{A}\colon  d(\mathcal{P},\mathcal{A})\ge \nu_2 \rbrace,  \\
\mathcal{E}_{\rm btw} &= \lbrace \mathcal{A}\colon  d(\mathcal{P},\mathcal{A})\in (\nu_1,\nu_2) \rbrace .
\end{align*}
Then, {\Fc by Theorem~\ref{thm:OGP} and Lemma~\ref{lem:opt_far_mod}} with probability at least $1-\eps,$ 
\begin{align}
&\arg\max 
q_{\mathcal{A}} 
\subseteq \mathcal{E}_{\rm close}, \label{eq:OGP-implications-1}\\
&\max_{\mathcal{A}: \mathcal{A}\in \mathcal{E}_{\rm far}} q_{\mathcal{A}}\ge q^{\star} - c_1, \label{eq:OGP-implications-2}\\
&\max_{\mathcal{A}: \mathcal{A}\in \mathcal{E}_{\rm btw}} q_{\mathcal{A}}\le q^{\star} - c_1-c_2. \label{eq:OGP-implications-3}
\end{align}

We now state our lower bound result.

\begin{theorem}\label{thm:slow_mixing_formal}
Consider the MCMC initiated at the Gibbs distribution conditioned on being in $\mathcal{E}_{\rm far}$. Namely, suppose
\begin{align*}
\pr\left(\mathcal{A}_0=\mathcal{A}\right) =  
{\pr_{\rm Gibbs}(\mathcal{A})  \over \sum_{\mathcal{B}\in \mathcal{E}_{\rm far}}  \pr_{\rm Gibbs}(\mathcal{B}) },
\end{align*}
for all $\mathcal{A}\in \mathcal{E}_{\rm far}$, and $\pr\left(\mathcal{A}_0=\mathcal{A}\right)=0$ otherwise.
Let 
$\tau=\min\lbrace t\colon  \mathcal{A}_t \in  \mathcal{E}_{\rm close}\rbrace$. There exist $c_3,c_4>0$ such that 
 with probability at least $1-\eps$ (with respect
to the randomness of the graph) 
\begin{align*}
\pr\left(\tau \ge \exp(c_3 n)\right) \ge 1-\exp(-c_4 n),
\end{align*}
where the probability is with respect to the random choices of the MCMC.
\end{theorem}

\begin{proof}
We assume that the events (\ref{eq:OGP-implications-1})-(\ref{eq:OGP-implications-3}) hold, which is the case with probability at least $1-\eps$.
We note that the events $\mathcal{A}_0\in \mathcal{E}_{\rm far}, \mathcal{A}_\tau\in \mathcal{E}_{\rm close}$
imply the existence of $s<\tau$ such that $\mathcal{A}_s\in \mathcal{E}_{\rm btw}$.  
For every positive integer $t$,
\begin{align*}
\pr\left( \mathcal{A}_t\in \mathcal{E}_{\rm btw} | \mathcal{A}_0\in \mathcal{E}_{\rm far} \right)  &=
{\pr\left( \mathcal{A}_t\in \mathcal{E}_{\rm btw} , \mathcal{A}_0\in \mathcal{E}_{\rm far} \right) 
\over
\pr\left(\mathcal{A}_0\in \mathcal{E}_{\rm far} \right) } \\
&\le 
{\pr\left( \mathcal{A}_t\in \mathcal{E}_{\rm btw}  \right)
\over
\pr\left(\mathcal{A}_0\in \mathcal{E}_{\rm far} \right) }. 
\end{align*}
We have
\begin{align*}
&\pr\left( \mathcal{A}_t\in \mathcal{E}_{\rm btw}\right)\le  k^n \exp(\beta n (q^{\star}-c_1-c_2)), \\
&\pr\left( \mathcal{A}_t\in \mathcal{E}_{\rm far}\right) \ge  \exp(\beta n (q^{\star}-c_1)).
\end{align*}
Thus the ratio is at most $k^n\exp(-\beta n c_2)$. Assuming $\beta$ is large enough so that $\log k-\beta c_2\le -1$ (again any negative
constant suffices), this ratio is at most $\exp(-n)$. By the union bound, we obtain that $\pr(\tau\le \exp(n/2))\le \mathrm{e}^{n/2}\mathrm{e}^{-n}=\mathrm{e}^{-n/2}$,
completing the proof.
\end{proof}

\section{Proofs of concentration results}\label{app:remainder_of_proof}\label{app:concentration}

In this section we provide the proof of Lemma~\ref{lem:write_mod_as_signature}. We prove this via an intermediate lemma, Lemma~\ref{lem:weightedSBMk}, on the modularity scores of partitions of a {\em weighted} graph corresponding to the SBM, which we prove first, in Appendix~\ref{subsec:weightedSBMk}.

\subsection{Intermediate result on weighted graphs}\label{subsec:weightedSBMk}

The definition for modularity extends naturally to \emph{weighted} graphs, and indeed is often used for clustering  weighted graphs in applications. Suppose that the edge-weights $w$ are such that
$0\leq w_{uv}=w_{vu} \leq 1$ for all nodes $u$ and $v$, and assume that $w$ is not identically zero.  
Define the (weighted) degree of a node $u$ by setting $d_u^{(w)}=\sum_{v\in V} w_{uv} $, the (weighted) volume of a node set $X$ by $\vol_w (X) = \sum_{u \in X} d_u^{(w)}$, and let $e_w(X)= \tfrac12 \sum_{u,v \in X} w_{uv}$ be the total edge-weight. 

For a given partition $\cA$ of $V$, define the \emph{modularity score} of $\cA$ on $w$ by \begin{eqnarray}
q_\cA(w) = 
\frac{1}{\vol_w(V)}\sum_{A\in \cA} \sum_{u,v \in A} \left(w_{uv} -\frac{d^{(w)}_u d^{(w)}_v}{\vol_w(V)}\right) 
= \sum_{A\in \cA} \frac {e_w(A)}{e_w(V)} - \sum_{A\in \cA} \left( \frac {\vol_w(A)}{\vol_w(V)} \right)^2, \label{eq:modularity_weights}
\end{eqnarray}
and the \emph{modularity} of $w$ by $q^{\star}(w)=\max_{\cA} q_{\cA}(w)$.
{\Fc As in the original modularity function, we express the modularity score as the difference between the edge contribution $q^E_\cA(w)$ and the degree tax $q^D_\cA(w)$ -- see~\eqref{def.qEandqD}.}
Note that if the edge weights in $w$ are $\{0,1\}$-valued, then~$q_{\cA}(w)$ is the usual modularity score. 

Given a probability weight function $w$, let $G_{w}$ be the random (unweighted) graph obtained by considering each pair of nodes $u,v$ independently, and including edge $uv$ with probability $w_{uv}$.
Lemma~\ref{lem:write_mod_as_signature} will be proven in the following way. We first calculate the modularity score of a partition with a given signature in a deterministic weighted graph (see also Lemma~\ref{lem:weightedSBMk}). Secondly, we show that the random graph $G \in \mathcal{G}(n,k,p,q)$ is likely to have the property that the modularity score of any $k$-part partition is within a small window of the  corresponding modularity score in the deterministic weighted graph. For the latter we use a concentration result from~\cite{sampling}, and this is done in Appendix~\ref{subsec:write_mod_as_signature}.

\begin{lemma}\label{lem:weightedSBMk}
Fix integer $k \ge 2$. For integer $n \ge k$ that is divisible by $k$, let $P_1 \cup \cdots \cup P_k$ be a partition of $V=[n]$, where $|P_j| = n/k$ for all $j \in [k]$. 
Let $w=w(n, k, p, q)$ be the weight function on node set $V$ with $w_{uv}=p$ if $u$ and $v$ are in the same part $P_j$ for some $j \in [k]$ and with $w_{uv}=q$ otherwise. Then, for all $k$-part partitions $\cA$ with signature $X=X(\cA)=(x_{ij})_{1 \le i,j \le k}$, 
\[
q_\cA(w) = (1+O(1/n)) \, \frac{p-q}{p+(k-1)q} \cdot \frac {g(X)}{k^2},
\] 
where $g(X)$ is defined as in~\eqref{defn.gx}.
\end{lemma}


\begin{proof}[Proof of Lemma~\ref{lem:weightedSBMk}.]
To simplify the notation, in this proof we will use $f(n) \sim f'(n)$ to denote that $f(n) = (1+O(1/n)) f'(n)$. 

First, note that the weighted degree is the same for each node, namely, for any $u \in V$, 
$$
d_u^{(w)} = (n-1) q + (n/k-1) (p-q) \sim n \left( q + \frac {p-q}{k} \right) = \frac {n}{k} \Big( p+(k-1)q \Big).
$$
As a result, $\vol_w(V) \sim \frac {n^2}{k} ( p+(k-1)q )$, and so $e_w(V) = \frac {1}{2} \vol_w(V) \sim \frac {n^2}{2k} ( p+(k-1)q )$.

Let us now fix $i \in [k]$ and aim to find an asymptotic value of $e_w(A_i)$. Note that there are two types of contributions to consider, namely, the weight of edges within each planted part $P_j$ and the weight of edges between planted parts $P_j$ and $P_{j'}$. First, note that
$$
|A_i| = \sum_{j=1}^k |A_i \cap P_j| = \frac {n}{k} \sum_{j=1}^k x_{ij}.
$$
(Recall the definition of the signature: $x_{ij} = |A_i \cap P_j| / (n/k)$.) We get that
\begin{eqnarray*}
e_w (A_i) &=& \binom{|A_i|}{2} q + \sum_{j=1}^k \binom{|A_i \cap P_j|}{2} (p-q) \\
&\sim& \frac {1}{2} \left( \left( \frac {n}{k} \sum_{j=1}^k x_{ij} \right)^2 q + \sum_{j=1}^k \left( \frac {n}{k} \, x_{ij} \right)^2 (p-q) \right) \\
&=& \frac {n^2}{2k^2} \left( q \left( \sum_{j=1}^k x_{ij} \right)^2 + (p-q) \sum_{j=1}^k x_{ij}^2 \right), 
\end{eqnarray*}
and so the edge contribution is equal to 
\begin{equation}\label{eq:edge_contr}
\sum_{i=1}^k \frac {e_w (A_i)}{e_w(V)} \sim \frac {1}{k} \sum_{i=1}^k \left( \frac {q}{p+(k-1)q} \left( \sum_{j=1}^k x_{ij} \right)^2 + \frac {p-q}{p+(k-1)q} \sum_{j=1}^k x_{ij}^2 \right).
\end{equation}
On the other hand, again using the fact that the weighted degree is the same for each node, 
$$
\vol_w(A_i) \sim |A_i| \frac {n}{k} \Big( p+(k-1)q \Big) = \left( \frac {n}{k} \sum_{j=1}^k x_{ij} \right) \frac {n}{k} \Big( p+(k-1)q \Big) = \frac {n^2}{k^2} \Big( p+(k-1)q \Big) \sum_{j=1}^k x_{ij},
$$
and so 
$$
\left( \frac {\vol_w(A_i)}{\vol_w(V)} \right)^2 \sim \left( \frac {1}{k} \sum_{j=1}^k x_{ij} \right)^2 = \frac {1}{k^2} \left( \sum_{j=1}^k x_{ij} \right)^2.
$$
We again use $\vol_w(V) \sim \frac {n^2}{k} ( p+(k-1)q )$ to conclude that the degree tax is equal to
\begin{equation}\label{eq:degree_tax}
\sum_{i=1}^k \left( \frac {\vol_w(A_i)}{\vol_w(V)} \right)^2 \sim \frac {1}{k^2} \sum_{i=1}^k \left( \sum_{j=1}^k x_{ij} \right)^2.
\end{equation}
Combining~(\ref{eq:edge_contr}) and~(\ref{eq:degree_tax}), we get that
\begin{eqnarray}
q_\cA(w) &=& \sum_{i=1}^k \frac {e_w (A_i)}{m} - \sum_{i=1}^k \left( \frac {\vol_w(A_i)}{\vol_w(V)} \right)^2 \nonumber \\
&\sim& \frac {1}{k} \sum_{i=1}^k \left( \left( \frac {q}{p+(k-1)q} - \frac {1}{k} \right) \left( \sum_{j=1}^k x_{ij} \right)^2 + \frac {p-q}{p+(k-1)q} \sum_{j=1}^k x_{ij}^2 \right) \nonumber \\
&=& \frac {1}{k^2} \cdot \frac {p-q}{p+(k-1)q} \sum_{i=1}^k \left( k \sum_{j=1}^k x_{ij}^2 - \left( \sum_{j=1}^k x_{ij} \right)^2 \right). \label{eq:almost_modularity}
\end{eqnarray}
Finally we note that 
\begin{eqnarray*}
k \sum_{j=1}^k x_{ij}^2 - \left( \sum_{j=1}^k x_{ij} \right)^2 &=& (k-1) \sum_{j=1}^k x_{ij}^2 - 2 \sum_{1 \le j<j' \le k} x_{ij} x_{ij'} \\
&=& \sum_{1 \le j<j' \le k} \left( x_{ij}^2 - 2 x_{ij} x_{ij'} + x_{ij'}^2 \right) \\
&=& \sum_{1 \le j<j' \le k} \left( x_{ij} - x_{ij'} \right)^2.
\end{eqnarray*}
Using this substitution in~(\ref{eq:almost_modularity}), we get the desired asymptotic value of $q_\cA(w)$ and the proof of the lemma is finished. 
\end{proof}

\subsection{Proof of Lemma~\ref{lem:write_mod_as_signature} - the concentration result}\label{subsec:write_mod_as_signature}
In this section, we present the proof of Lemma~\ref{lem:write_mod_as_signature}.

\label{proof.concentration}
We use the following proposition which follows from Theorem~10.4 and its proof in~\cite{sampling}.
We say a partition $\cA$ is $\eta$--fat for $w$ if $\vol_w(A)\geq \eta \vol_w(V)$ for all parts $A \in \cA$. We write $\cA_{\eta\;  {\rm small}}$ for the set of parts $A\in\cA$ with $\vol(A)<\eta \vol_w(V)$ and call such sets $\eta$-small.%
\begin{proposition}
\label{prop:from_sampling}
Given $\eps>0$, there exists $c=c(\eps)$ such that the following holds. Let $w$ satisfy $e_w(V)\geq c$. Then, with probability at least $1-\eps$, for any partition~$\cA$ which is ($\eps/4$)--fat for $w$,   
\[
|q_\cA(G)-q_\cA(w)| \leq \eps.
\]
\end{proposition}
It will also be useful to use the next lemma, 
which allows us to restrict our attention only to these $\eta$-fat partitions where all parts have large volume. See Lemma~3.2 and its weighted counterpart in~\cite{sampling} for the definition of the amalgamating algorithm, and results on it.   
\begin{lemma}\label{lem.fattening_w}
For each $w$ on $E$, and each $0<\eta \leq 1$, there is an $\eta$-fat partition $\cB$ of $V$ such that $q_{\cB}(w) > \q(w) - 2\eta$. Indeed, given any partition $\cA$ of~$V$, the greedy amalgamating algorithm uses a linear number of operations and constructs an $\eta$-fat partition $\cB =\cA(G,\eta,\cA)$ such that\[ 2\zeta < 
q_{\cA}(w) -  q_{\cB}(w) < 2\eta,\]
where $\zeta = \sum_{A\in \cA_{\eta \; \rm{small}}} \vol_w(A)/\vol_w(V)$. 
\end{lemma}

Note that $\zeta$ is the proportion of the volume contained in low-volume parts of $\cA$.
The following robustness result will allow us to ignore any small proportion of edges~\cite{sampling}.
\begin{lemma}\label{lem:robustness}
Let $H=(V,E)$ be a graph and let $\cA$ be a partition of $V$.
If $E_0$ is a non-empty proper subset of $E$, and $H'=(V, E \setminus E_0)$, then 
\[ |q_\cA(H)- q_\cA(H')| \; 
< \frac{2 \, |E_0|}{|E|}. \]
\end{lemma}

Now, we are ready to prove Lemma~\ref{lem:write_mod_as_signature}.

\begin{proof}[Proof of Lemma~\ref{lem:write_mod_as_signature}]
Clearly, with probability at least $1-\eps/10$, there are $k-1$ nodes of degree at most $2ca = O(1)$. Removing at most $k-1$ of such nodes changes the volume (and so also the number of edges) by $O(1)$. Hence, by Lemma~\ref{lem:robustness}, for any partition $\cA$, removing such nodes changes the modularity score by $O(1/n) \le \eps/3$. 
We may then assume that $n$ is divisible by $k$. 

Suppose $G$ and $\eps>0$ are such that in Proposition~\ref{prop:from_sampling}, $|q_\cA(G)-q_\cA(\omega)|<\eps/5$ for any $\eps/20$-fat partition with probability at least $1-\eps/5$.
Take an arbitrary $k$-part $\cA$. Then, it will suffice to show that 
\[ 
\left| q_\cA(G) - \frac{a-b}{(a+(k-1)b)} \cdot \frac {g(X)}{k^2} \right| < \eps, 
\]
where $X = X(\cA) =(x_{ij})_{1\leq i,j \leq k}$ is the signature of $\cA$ and $g(X)$ is as defined in~\eqref{defn.gx}.

Let $\eta=\eps/8(k-1)$. The easy case is that $\cA$ is an $\eta$-fat partition for $\omega$. Then, by Proposition~\ref{prop:from_sampling}, with probability $1-\eps/10$, we have $|q_\cA(G)-q_\cA(w)|\leq \eps/10$, but then, by Lemma~\ref{lem:weightedSBMk}, we are done.

Now suppose that $\cA$ is not an $\eta$-fat partition of $w$. 
Note that at most $k-1$ parts may have volume at most $\eta \, \vol_w(V)$ since $\eta<1/k$. Thus, by Lemma~\ref{lem.fattening_w} there is an $\eta$-fat partition $\cB$ such that
\begin{equation}\label{eq.det_mod_close}  
2(k-1)\eta \leq q_\cA(w) - q_{\cB}(w) \leq 2\eta. 
\end{equation}
Hence by our choice of $\eps$, $|q_\cA(w) - q_{\cB}(w)|\leq \eps/4$.
Again, by Proposition~\ref{prop:from_sampling}, we have $|q_{\cA}(G) - q_{\cB}(w)|\leq \eps/4$. 
Note that~\eqref{eq.det_mod_close} together with Lemma~\ref{lem:weightedSBMk} implies that that $|g(X)-g(X_B)|$ is sufficiently small, where $X_B$ corresponds to $\cB$ and $X$ corresponds to $\cA$. Then, a sequence of triangle inequalities gives us the required bound.
\end{proof}

\section{Optimisation lemmas for modularity at a given distance}
\label{app:optimisation}

In this section we prove that the function $g(X)$ has the required behaviour. Recall that by Theorem~\ref{thm:main_theorem} with probability close to one, the modularity of any partition $\cA$ with signature $X$ can be well approximated by $\frac{a-b}{a+(k-1)b}\,\frac{1}{k^2}g(X)$. 

In particular, we are interested in the maximum modularity score for partitions at a given distance $d$; and this we can control by understanding the maximum of $g(X)$ for signatures $X$ at distance $d$; or, equivalently, the maximum of $\frac{1}{k^2}\,g(X)$, which we call $\tilde{h}(d)$. (See also the discussion after Theorem~\ref{thm:main_theorem}.) This will correspond to the maximisation problem $\tilde{h}(d)=\frac{1}{k^2}\,\max_{X \in \mathcal{X}_k(dk)} g(X)$ for the $k$-block SBM, and re-scaled distance $t=dk$.

These results naturally split into two. For $d<1/k$, we solve this optimisation in Lemma~\ref{lem:optimisation}, while, for $d>1/k$, we show in Lemma~\ref{lem:opt_far_mod} that $h(d)<h(1/k)$, i.e., that the value of $g$ is bounded above by its maximum at $d=1/k$.

\subsection{Proof of the general case for distance $d<1/k$}
\label{app:optimisation-lemma}

\begin{proof}[Proof of Lemma~\ref{lem:optimisation}] Let $t=dk$ and consider first $0<t<1$. We will prove the result by successively eliminating the variables.

We begin by giving an equivalent optimisation problem with $k(k-1)$ variables. Since the columns of the matrix $X$ sum to $1$, we may reduce the number of variables by writing $x_{ii}=1-\sum_{i':i'\neq i} x_{i'i}$, and maximising the resulting $\tilde{g}(x_{12}, \ldots, x_{k-1,\; k})$ obtained from $g(X)$ by making these substitutions for $x_{ii}$. 

In particular, let
\begin{equation}
\widetilde{\mathcal{X}}=\big\{(x_{ij})_{i\neq j}
~\colon~ x_{ij}\geq 0 \; \forall i\neq j, \; \sum_{i\colon i\neq j}x_{ij}\leq 1 \; \forall j \; \quad\mbox{and} \quad\;
\label{eq.small_d.X}
\sum_{i\ne j} x_{ij}=t\big\}. \end{equation}
Thus it will suffice to show $\max_{\underline{x} \in \widetilde{\mathcal{X}}} \tilde{g}(\underline{x})$ has the bounds claimed. (We write $\underline{x}$, since we have now a subset of variables of the matrix $X$.)

Notice that 
if we have equality for $x_{ij}+x_{i'j}= 1$ for some $i,i',j$ all distinct, then $\sum_{i\ne j}x_{ij}\geq 1$, which yields a contradiction since $t<1$. Thus, it is equivalent to consider the maximisation problem over $\widetilde{\mathcal{X}}_1$, where we require strict inequality for the terms, i.e., $x_{ij}+x_{i'j}< 1$. Let
\begin{equation}
\widetilde{\mathcal{X}}_1=\big\{(x_{ij})_{i\neq j}
~\colon~  x_{ij}\geq 0 \; \forall i\neq j, \; \sum_{i\colon i\neq j}x_{ij} < 1 \; \forall j \; \quad\mbox{and}\quad \;
\label{eq.small_d.X1}
\sum_{i \ne j} x_{ij}=t\big\}. \end{equation}
It now suffices to show that $\max_{\underline{x} \in \widetilde{\mathcal{X}}_1} \tilde{g}(\underline{x})$ satisfies the claimed bounds.

We may remove one more variable. Note
from the constraints that we have \begin{align*}    
x_{21}&=t-\sum_{{i \ne j\atop (i,j)\ne (2,1)}}x_{ij}.\end{align*}
Let $\tilde{g}_2$ be $\tilde{g}$ making this substitution (i.e.,\ $\tilde{g}_2$ is a function on $k^2-k-1$ variables), and let 
\begin{eqnarray}
\nonumber 
\widetilde{\mathcal{X}}_2=\big\{(x_{ij})_{i\neq j,\; (i,j)\neq (2,1)}
& \colon &  x_{ij}\geq 0 \; \forall (i,j) \text{ with } i\neq j, (i,j)\neq (1,2), \;  \sum_{i\colon i \geq 3}x_{i1} < 1 \\
&& \sum_{i\colon i\neq j}x_{ij} < 1 \; \forall j\geq 2  \; \quad \mbox{and}\quad  \;
\label{eq.small_d.X2}
\sum_{{i \ne j\atop (i,j)\ne (2,1)}} x_{ij} \leq t\big\}. \end{eqnarray}
Thus, finally, it suffices to show that $\max_{\underline{x} \in \widetilde{\mathcal{X}}_2} \tilde{g}_2(\underline{x})$ satisfies the claimed bounds.

Recall that the maximum of a convex function is obtained on the boundary. Note also that we may apply this recursively, since if we set any inequality in~\eqref{eq.small_d.X2} to equality, then we may make a substitution to reduce the number of variables by $1$, and we obtain a maximisation problem for a convex function (on one fewer variables). Thus, the max of $\tilde{g}_2(\underline{x})$ over $\underline{x}\in \widetilde{\mathcal{X}}_2$ is obtained when at least $k^2-k-1$ of the inequalities in~\eqref{eq.small_d.X2} are equality. 

Observe that there are $k^2-k-1$ inequalities in~\eqref{eq.small_d.X2} (which are not strict inequalities), so the maximum is obtained when we set all but (at most) one of these to equality.

We consider two cases. The first case is that we have $x_{ij}=0$ for all $(i,j)$ with $i\neq j,\; (i,j)\neq (2,1)$. Notice that this fixes the value of \emph{all} variables, and that the remaining inequalities in~\eqref{eq.small_d.X2} are all satisfied.
The value of  $\tilde{g}_2(x)$ attained is that of $g(X)$ with $x_{21}=t$ (and  $x_{ij}=0$ for all $(i,j)$ with $i\neq j,\; (i,j)\neq (2,1)$, and $x_{ii}=\sum_{i' \colon i'\neq i} x_{i' i}$ for all $i$). Note that this yields the matrix $X^{(2,1)}(t)$, and, moreover, that
$g(X^{(2,1)}(t))= k(k\!-\!1) - 2t \big( k\!-\!t(k\!-\!1)\big)$.

The second case is that there exist $a,b$ (where $(a,b)\neq (2,1)$) such that 
\[ \sum_{{i \ne j\atop (i,j)\ne (2,1)}} x_{ij} = t \; \mbox{ and }\; \mbox{
$x_{ij}=0$ for all $(i,j)$ with $i\neq j,\; (i,j)\neq (2,1), (a,b)$}.\] 
Notice that this fixes the value of all $k^2-k-1$ variables: $x_{ab}=t$ and all remaining variables take the value $0$, and that this assignment satisfies the remaining inequalities in~\eqref{eq.small_d.X2}. The value of  $\tilde{g}_2(x)$, attained is that of $g(X)$ with $x_{ab}=t$ (and  $x_{ij}=0$ for all $(i,j)$ with $i\neq j,\; (i,j)\neq (a,b)$, and $x_{ii}=\sum_{i' : i'\neq i} x_{i' i}$ for all $i$). Note that this yields the matrix $X^{(a,b)}(t)$, and, moreover, that again $g(X^{(a,b)}(t))= k(k\!-\!1) - 2t \big( k\!-\!t(k\!-\!1)\big)$ for any such $(a,b)$.
This completes the proof. 
\end{proof}

\subsection{Proof for the general case $d>1/k$}\label{subsec:far_mod}

Here we prove Lemma~\ref{lem:opt_far_mod}, which will allow us to infer that partitions $\cA$ with high modularity scores must be at distance $d\leq 1/k$ from the planted partition.

\begin{proof}
Consider any partition $\cA$ with signature $X=(x_{ij})_{1\leq i \leq j \leq k}$ that is at distance more than $1/k$ from the planted partition $\cP$. Our goal is to show that $g(X) \le k(k-1)-2$.

Suppose that there exists a column $j \in [k]$ (say, $j=1$) in $X$ with at least two non-zero entries (say, $x_{11}>0$ and $x_{21}>0$). Consider a family of signatures $X(s)$ that is parametrised by variable $s$, as follows: for a given $s \in [-x_{11}, x_{21}]$, matrix $X(s)$ is exactly the same as $X$ but $x_{11}$ is replaced with $x_{11}+s$ and $x_{21}$ is replaced with $x_{21}-s$. Of course, $X = X(0)$.

Note that $g(X(s))$ is a quadratic function of $s$ with a positive coefficient in front of $s^2$. As a result, at least one of the following two properties holds:
(i) $g(X(s))$ increases as $s$ increases from $s=0$ to $s=x_{21}$ (which is equivalent to transferring a weight from $x_{21}$ to $x_{11}$); or
(ii) $g(X(s))$ increases as $s$ decreases from $s=0$ to $s=-x_{11}$ (which is equivalent to transferring a weight from $x_{11}$ to $x_{21}$). 
Note that it might be the case that $g(X(s))$ attains its local minimum at $s=0$ and both properties hold at the same time. 
Regardless, we may start from $s=0$ and either increase or decrease $s$ to gradually increase $g(X(s))$. While we do this, the distance to the planted partition $\cP$ does not need to behave monotonically. However, it is easy to see (but it is crucial for the argument) that the distance to $\cP$ is a continuous function of $s$. Once we are done with transferring the weight, we get an additional zero entry in our signature matrix, and we can move on to the next pair of non-zero entries, possibly in a different column $j$. 

We need to consider two cases now. Suppose first that during the above process the distance to $\cP$ is equal to $1/k$. We prematurely stop the process at the very first time this happens, and let $Y$ be the signature we have at that moment. Since we kept increasing the function $g$ along the way, it follows from Lemma~\ref{lem:optimisation} (applied with $t=1$) that 
$$
g(X) < g(Y) \le \max_{X\in \mathcal{X}_k(1)} g(X) = k(k-1)-2,
$$
which finishes the proof of the lemma in this case. 

Suppose now that the distance from $\cA$ to $\cP$ was always more than $1/k$, but we had to stop the above process of transferring weights at some point because each column had exactly one non-zero entry (of course these non-zero entries must be equal to one). This means that each part of $\cA$ is a union of some planted parts, and the distance between $\cA$ and $\cP$ is equal to $i/k$ for some integer $i \ge 2$. 

Now, take any part which is a union of $r \ge 2$ planted parts and split it into two parts, consisting of $r-1$ and a single planted part, respectively. After such partition refinement, the function $g$ increases. Indeed, before splitting the part with $r$ planted parts, the contribution of that part to the function $g$ is $r(k-r)$, but after splitting, the two resulting parts contribute $(r-1)(k-r+1)+(k-1)$, which is equal to $r(k-r) + 2(r-1) > r(k-r)$. We continue such refinements of partitions until the distance to $\cP$ is $1/k$. At that point, the partition $\cA$ with signature $Y$ has one part consisting of two planted parts, while the remaining parts consist of just one planted part. The conclusion is as before: 
$$
g(X) < g(Y) = \max_{X\in \mathcal{X}_k(1)} g(X) = k(k-1)-2.
$$
This finishes the proof of Lemma~\ref{lem:opt_far_mod}. 
\end{proof}

\section{Proof of Theorem~\ref{thm:OGP} : SBM has OGP}

\label{app:proof-OGP}

\begin{proof}[Proof of Theorem~\ref{thm:OGP}]\label{proof.OGP} We will prove Theorem~\ref{thm:OGP} using Theorem~\ref{thm:main_theorem}, our concentration result in  Lemma~\ref{lem:write_mod_as_signature}, and the optimisation result in Lemma~\ref{lem:opt_far_mod}.

First, we will set up the range for the distances. Recall that $h(d)$ is a quadratic function of $d$, with minimum at $\frac{1}{2(k-1)}$. See also Figure~\ref{fig:functions_h_new}. Take any $\nu\in \left(\frac{1}{2(k-1)}, \frac{1}{k}\right)$, choose 
$\nu'= \frac{2}{3}\cdot \nu + \frac{1}{3}\cdot \frac{1}{k}$, and let $\nu''\in \left[0,\frac{1}{2(k-1)}\right)$ be the unique number such that $h(\nu'')= h(\nu')$. Let $n_0$ be the smallest integer such that $1/\sqrt{n_0}< \nu'-\nu< \frac{1}{2(k-1)}-\nu''$. 

Next, we define $\mu=\frac {a-b}{a+(k-1)b}\big(h(0)-h(\nu')\big)$ and 
note that $\frac {a-b}{a+(k-1)b}\,h(0)-\mu =  \frac {a-b}{a+(k-1)b}\,h(\nu')$. Choose a positive $\varepsilon<\frac{a-b}{a+(k-1)b}\,(h(\nu')- h(\nu))$. Then, by Theorem~\ref{thm:main_theorem}, we obtain that parts (i) and (ii) below hold with probability at least $1-\varepsilon$, and for $n>n_0$ large enough:
\begin{itemize}
    \item[(i)] For all partitions $\cA$ at distance $d\in [0,\nu'-\nu)$ and at distance $d\in (1/k-(\nu'-\nu),1/k]$, it holds that 
    \[q_{\cA}(G)>\frac {a-b}{a+(k-1)b}\,h(d)-\varepsilon > \frac {a-b}{a+(k-1)b}\,h(\nu')=\frac {a-b}{a+(k-1)b}\,h(0)-\mu.\]    
    \item[(ii)] For any partition $\cA$ at distance $d\in \left(\frac{k}{2(k-1)},\nu\right)\subset\left[0,\frac{1}{k}\right]$, 
    \[q_{\cA}(G)< \frac {a-b}{a+(k-1)b}\,h(d)+\varepsilon< \frac {a-b}{a+(k-1)b}\,h(\nu')=\frac {a-b}{a+(k-1)b}\,h(0)-\mu.\]    
 
\end{itemize}
Together, (i) and (ii) establish Theorem~\ref{thm:OGP}.
\end{proof}

\section{Proof of Theorem~\ref{thm.close_is_close}}\label{app:close_is_close}

We prove the following more detailed proposition, which implies Theorem~\ref{thm.close_is_close}.

\begin{proposition}    
[\bf Partitions with near-optimal modularity are close to the planted]\label{prop.detailed_close_is_close} 

Fix $a>b>0$ and integer $k\ge 2$. Let $p=p(n)=\omega\; a / n$ and $q=q(n)=\omega \; b/n$. For any $\eps>0$ and $\delta<\frac{a-b}{a+(k-1)b} \cdot \frac{2}{k^2}$, there exists $c=c(\eps)$ such that, if $\omega>c$, then the following holds with probability at least $1-\eps$ provided $n$ is large enough.

Let $G \in \mathcal{G}(n,k,p,q)$ with the planted partition $\cP$. 
Then, $q_\cA(G)>\q(G) - \delta $ for some $k$-part partition $\cA$ implies $d(\cA, \cP) < \delta' + \varepsilon$,
where $\delta' =\delta'(\delta) \in \big(0,\frac{1}{k(k-1)}\big)$.

In fact, $\delta'=\delta'(\delta)$ is the unique value in the interval $\big(0,\frac{1}{k(k-1)}\big)$ that satisfies 
$$
\frac {a-b}{a+(k-1)b} h(\delta') =  \frac {a-b}{a+(k-1)b} h(0)-\delta,
$$
which is equivalent to 
$$
\delta = \frac {a-b}{a+(k-1)b} \Big( 2 \delta' ( 1 - \delta'(k-1)) \Big).
$$
\end{proposition}

The intuition for the proof is that {\em if} the modularity score of a partition is sufficiently high, i.e., at least $\q(G)-\delta$, {\em then}, by our concentration results, we may translate this to a lower bound on $g(X)$ for $X$ the signature of $\cA$. Since Lemma~\ref{lem:opt_far_mod} upper bounds the possible values of $g$ for signatures far from the planted partition, we may then infer that $X$ must be within distance $d\leq 1/k$ of the planted partition. Then, by our detailed results in Theorem~\ref{thm:main_theorem}, we may conclude that the distance of $X$ to the planted partition is within $\delta'$ since $g(X)$ attains such a high value.
We now provide the details.

\begin{proof} Let $s_\delta=\delta \frac{a+(k-1)b}{a-b}$, and pick $\eps < \frac{2}{k^2}-\delta$.
We first note that, by Lemma~\ref{lem:write_mod_as_signature}, we may take $\omega$ large enough so that with probability at least $1-\eps/4k^2$ for all  partitions $\cA$, we have that for the corresponding signature $X$,
\[ |q_\cA(G) - \frac{a-b}{a+(k-1)b}\frac{g(X)}{k^2} |\leq \frac{\eps}{4k^2}.\]
Recall also that with probability at least $1-\eps/4k^2$, \[ |\q(G) - \frac{a-b}{a+(k-1)b}\big(1-\frac{1}{k}\big)| \leq \eps /4k^2, \]
and thus, with probability at least $1-\eps/4$, for $q_\cA(G)>\q(G)-\delta$,
\[ g(X)> k(k-1) -k^2\delta -  \eps/2.\]
Hence, since $\delta \leq 2/k^2$, with probability at least $1-\eps/4$, by our choice of $\eps$,
\[ g(X)> k^2 - k -2.\]
Therefore, by Lemma~\ref{lem:opt_far_mod}, 
we may conclude that with probability at least $1-\eps/4$, we have $d(\cA, \cP)<1/k$ for the partition $\cA$ with signature $X$. The result now follows by applying Theorem~\ref{thm:main_theorem}. 
\end{proof}

\section{The case of balanced partitions}
\label{app:balanced}
Theorem~\ref{thm:main_theorem} shows that the OGP holds when optimising over all $k$-part partitions. If, instead of optimising over \emph{all} $k$-part partitions at a certain distance from the planted partition, we optimise only over \emph{balanced} such partitions, then the behaviour is markedly different. See the discussion in Section~\ref{sec.conj}. We now prove Proposition~\ref{prop.balanced} which we restate below.
\propbalanced*

For the proof we recall the notion of circulation on directed graphs and the following classic result on cycle decompositions of such graphs -- see for example \cite[Theorem 1]{gauthier2014decomposition} or \cite[Theorem 11.1 and (11.3)]{Schr03a}. %
Let $G=(V(G),\vec{E}(G))$ be a directed graph. A function $b:\vec{E}(G) \rightarrow \mathbb{R}_{\geq 0}$ is a \emph{circulation} on $G$ if for every vertex $v\in V(G)$ it holds that
\[ 
\sum_{\substack{ 
u\; : \; \vec{vu}\in \vec{E}(G)}} b_{vu}
\;=\;
\sum_{\substack{ 
u \; : \; \vec{uv}\in \vec{E}(G)}} b_{uv},
\]
i.e., the total outflow from $v$ equals the total inflow into $v$. Circulations have the following structure related to directed cycles:
\begin{lemma}[Cycle decomposition of a nonnegative circulation]
\label{lem:cycle_decomposition}
Suppose $B=(b_{ij})_{i,j\in[k]}$ is a circulation and satisfies $b_{ij}\ge 0$ for $i\neq j$ and $b_{ii}=0$.
Then, for some $m \leq |(i,j) : b_{ij}>0|$,
 there exist \emph{simple directed cycles} 
$C_1,\dots,C_m$ in $G$ and weights $w_1,\dots,w_m>0$ such that, for every $i\neq j$,
\[
b_{ij} \;=\; \sum_{\ell=1}^m w_\ell \,\mathbf{1}\{ \vec{ij}\in C_\ell\}.
\]
\end{lemma}

\begin{proof}[Proof of Proposition~\ref{prop.balanced}]
Fix $0<t_1 < t_2 < k-1$, and let $X^*\in \cX'_k(t_2)$ be such that $g(X^*)=g^{\rm bal}(t_2)$. Note that it suffices to construct $\widetilde{X}\in \cX'_k(t_1)$ such that $g(\widetilde{X})>g(X^*)$ since then we would have
\begin{equation}\label{eq.whyXtilde} g^{\rm bal}(t_1) \geq g(\widetilde{X})>g(X^*) = g^{\rm bal}(t_2).\end{equation}
The intuition for the proof is that we can regard the non-diagonal entries of $X^*$ as a directed graph; and they form a non-negative circulation. We can then apply Lemma~\ref{lem:cycle_decomposition} to get a directed cycle decomposition $\{(C_\ell,w_\ell)\}_\ell$ such that these sum to the off-diagonal entries of $X^*$. For any directed cycle and $\eps \leq w_\ell$, we may modify an $X\in \cX'(t)$ by deleting an $\eps$ weight from any $x_{ij}$ with $\vec{ij}$ in the cycle, adding $\eps$ weight to $x_{ii}$. We call this a $(C, \eps)$-transfer move, and we will see that this preserves the properties in~\eqref{eq.def.X.orig}, and will thus yield an element of $\cX(t')$ for $t'<t$. We will show that the $g$ value has strictly decreased by this operation. Starting from $X^*$, we use the cycle decomposition to perform a sequence of transfer moves which allows us to construct the required $\widetilde{X}$ in~\eqref{eq.whyXtilde}. Let us now give the details.

We begin by defining the $(C,\eps)$-transfer move on matrix $X$. Given a directed cycle $C$ and $\eps>0$ define $X^{(\eps)}=T_{C,\eps}(X)$ by setting 
\begin{align*}
x^{(\varepsilon)}_{ii} &:= x_{ii}+\varepsilon, \qquad \text{ for all }i\in V(C),\qquad \qquad x^{(\varepsilon)}_{ij} := x_{ij}-\varepsilon, \qquad \text{ for all }\vec{ij}\in \vec{E}(C),
\end{align*}
and leaving all other entries unchanged.

\begin{claim}\label{clm.ginc_stillvalid}
    We \emph{claim} that if $X\in \cX'_k(t)$, $C$ is a directed cycle with $V(C)\subseteq [k]$ and $0< \eps \leq \min_{\vec{ij}\in \vec{E}(C)} x_{ij}$, then, for $X^{(\eps)}=T_{C,\eps}(X)$,
    \begin{itemize}
        \item $g(X^{(\eps)})>g(X)$ and
        \item $X^{(\eps)}\in \cX'_t(t')$ where $t'=t-|C|\eps$.
    \end{itemize}
\end{claim}
\begin{claimproof}
We first prove that $g$ is strictly increased by the transfer move, i.e., $g(X^{(\eps)})>g(X)$. Notice that we may rewrite $g(X)$ using the Frobenius norm. Fix a row $i$, then, since $\sum_j x_{ij}=1$,
\[
\sum_{1\le j<j'\le k} (x_{ij}-x_{ij'})^2
=\frac12\sum_{j,j'}(x_{ij}-x_{ij'})^2
= k\sum_{j=1}^k x_{ij}^2 -\Big(\sum_{j=1}^k x_{ij}\Big)^2
= k\sum_{j=1}^k x_{ij}^2 - 1,
\]
so that
\begin{equation}
\label{eq:g_frob_bal}
g(X)
= k\sum_{i,j} x_{ij}^2 -k 
\end{equation}
Hence, showing $g(X^{(\eps)})>g(X)$ is equivalent to showing $\sum_{i,j}(x^{(\eps)}_{ij})^2> \sum_{i,j}x_{ij}^2$.
Note that only $2|C|$ entries change in the transfer move. Therefore,
\begin{align}
\nonumber
\sum_{i,j} \big(x^{(\varepsilon)}_{ij}\big)^2 - \sum_{i,j} x_{ij}^2 
&= \sum_{i \in V(C)} \Big((x_{ii}+\varepsilon)^2-x_{ii}^2\Big)
   +\sum_{\vec{ij}\in \vec{E}(C)} \Big((x_{ij}-\varepsilon)^2-x_{ij}^2\Big)\\
&= 2\varepsilon\Big(\sum_{i \in V(C)} x_{ii} - \sum_{\vec{ij}\in \vec{E}(C)} x_{ij}\Big) + 2|C|\varepsilon^2.\label{eq:diff_after_transfer}
\end{align}
To see that~\eqref{eq:diff_after_transfer} is positive, we argue as follows. Let $\pi\in S_k$ be the permutation that maps $j\mapsto i$ for directed edges $\vec{ij}$ on the cycle, and fixes all other vertices.
Then $\sum_j x_{\pi(j),j}=\sum_{\vec{ij}\in \vec{E}(C)}  x_{ij} + \sum_{i\notin V(C)}x_{ii}$.
However, by the alignment constraint of~\eqref{eq.def.X.orig}, which states that $\sum_{i=1}^k x_{ii} \geq \sum_{i=1}^k x_{\sigma(i)i}$ for all permutations $\sigma$, we have that
$\sum_{i=1}^k x_{ii} \geq \sum_{i=1}^k x_{\pi(i)i}$, so that
\[
\sum_{i\in V(C)}^k x_{ii} = \sum_{i=1}^k x_{ii} - \sum_{i\notin V(C)}^k x_{ii}  \;\ge\; \sum_{j=1}^k x_{\pi(j),j} - \sum_{i\notin V(C)}^k x_{ii} = \sum_{\vec{ij}\in \vec{E}(C)}  x_{ij}.
\]
Thus, \eqref{eq:diff_after_transfer} is at least $2|C|\eps^2>0$, so we have $g\big(X^{(\varepsilon)}\big) > g(X)$, as required.

To see the second part of the claim, note first that $x^{(\varepsilon)}_{ij}\ge 0$ for all $i,j\in [k]$ by construction.
Also, row sums remain $1$ since, for rows $i$ with $i\in V(C)$, we decrease one entry by $\eps$ and increase the diagonal
entry by the same amount, and other row sums do not change. The same argument applies to column sums, 
and thus $X^{(\varepsilon)}$ is doubly stochastic. Moreover, 
\[
\sum_{i\neq j} x^{(\varepsilon)}_{ij} \;=\; \sum_{i\neq j} x_{ij} - |C|\varepsilon \;=\; t-|C|\varepsilon=t'.
\]
Lastly, for any permutation $\sigma$, {\Fc we note that by the definition of the transfer move, for $i$ such that $\sigma(i)\neq i$, it holds that $x_{ii}^{(\eps)}\geq x_{ii}$ and $x^{(\eps)}_{\sigma(i)i}\leq x_{\sigma(i)i}$. Hence we have
\[\sum_i x^{(\eps)}_{ii}-\sum_{i}x^{(\eps)}_{\sigma(i)i} = \sum_{i \; : \; \sigma(i)\neq i} ( x_{ii}^{(\eps)} - x^{(\eps)}_{\sigma(i)i} ).\]
 }
Thus since $\sum_i x_{ii}\ge \sum_i x_{\sigma(i),i}$ held for $X$, it still holds for $X^{(\varepsilon)}$. We conclude that $X^{(\eps)}$ satisfies the constraints~\eqref{eq.def.X.orig} and $X^{(\eps)}\in \cX'(t')$ which completes the proof of the claim.
\end{claimproof}

We now return to constructing $\widetilde{X}$ to establish~\eqref{eq.whyXtilde}. Recall that $X^*\in\cX_k(t_2)$ is such that $g(X^*)=g^{\rm bal}(t_2)$.
Define $B=(b_{ij})$ by $b_{ij}=x^*_{ij}$ for $i\neq j$ and $b_{ii}=0$; and note that $B$ is a circulation on $[k]$. To see this we need to check that the outflow from any vertex $i\in[k]$ equals the total inflow into $i$. Since the $i$th row and column both sum to $1$ in $X^*$,
\[
\sum_{j\; : \; i\neq j} b_{ij}
\;=\;
\sum_{j\; : \; i\neq j} x^*_{ij}
\;=\;
1-x^*_{ii}
\;=\;
\sum_{j\; :i \neq j} x^*_{ji}
\;=\;
\sum_{j\; :i \neq j} b_{ji},
\]
as required, confirming that $B$ is a circulation. 
Therefore, by Lemma~\ref{lem:cycle_decomposition}, there exist simple directed cycles $C_1,\dots,C_m$ in $G$
and weights $w_1,\dots,w_m>0$ such that for every $i\neq j$,
\[
b_{ij} \;=\; \sum_{\ell=1}^m w_\ell\,\mathbf{1}\{\vec{ij}\in C_\ell\}.
\]
In particular, the total off-diagonal mass decomposes as
\begin{equation}
\label{eq:cycle_capacity_sum}
t_2 \;=\; \sum_{i\neq j} x^*_{ij}
\;=\; \sum_{i\neq j} b_{ij}
\;=\; \sum_{\ell=1}^m w_\ell \, |C_\ell|.
\end{equation}
Hence, we may choose $L$ such that $\sum_{\ell< L} w_\ell \, |C_\ell| < t_2 - t_1 \le \sum_{\ell\le L}  w_\ell \, | C_\ell|$.
Define $\varepsilon_\ell:=w_\ell$ for $\ell<L$, $\varepsilon_\ell:=0$ for $\ell>L$, and
\[
\varepsilon_L \ :=\ \frac{t_2 - t_1 - \sum_{\ell< L} w_\ell \, |C_\ell| }{r_L},
\]
so that $0<\varepsilon_L\le w_L$ and $\sum_{\ell=1}^L  \varepsilon_\ell \, |C_\ell| =t_2 - t_1$.
Now define a sequence of matrices by $X^{(0)}:=X^*$ and, for $\ell=1,\dots,L$,
\[
X^{(\ell)}\ :=\ T_{C_\ell,\varepsilon_\ell}\big(X^{(\ell-1)}\big) \quad\mbox{and}\quad \widetilde{X}:=X^{(L)}
\]
By Claim~\ref{clm.ginc_stillvalid}, each transfer move preserves nonnegativity, double stochasticity, the alignment constraint,
and decreases the off-diagonal mass by exactly $|C_\ell|\, \varepsilon_\ell$. Hence for $\widetilde{X}=X^{(L)}$ the off-diagonal mass is
\[
\sum_{i\neq j} x^{(L)}_{ij}
=
\sum_{i\neq j} x^{(0)}_{ij} - \sum_{\ell=1}^L |C_\ell|\, \varepsilon_\ell
=
t_2-(t_2-1)
=
t_1,
\]
and therefore $\widetilde X\in \cX_k'(t_1)$.
Moreover, again by Claim~\ref{clm.ginc_stillvalid}, each transfer move strictly increases $g$, so
\[
g(\widetilde X)
=
g\big(X^{(L)}\big)
>
g\big(X^{(0)}\big)
=
g(X^*),
\]
which establishes~\eqref{eq.whyXtilde}, and we are done.
\end{proof}

\end{document}